\documentclass[12pt,reqno]{amsart}
\usepackage{amssymb, amsmath, hyperref}
\usepackage{pdfsync}
\usepackage{stmaryrd}
\usepackage{bbold}
\usepackage{bbm}

\usepackage{graphics,graphicx,fancyhdr,color}

\newtheorem{df}{Definition}[section]
\newtheorem{lemma}[df]{Lemma}
\newtheorem{prop}[df]{Proposition}
\newtheorem{thm}[df]{Theorem}
\newtheorem{cor}[df]{Corollary}

\makeatletter \@addtoreset{equation}{section}

\newcommand{\coloneqq}{:=}

\newcommand{\cal}{\mathcal}

\newcommand{\bes}{\begin{displaymath}}
\newcommand{\ees}{\end{displaymath}}
\newcommand{\be}{\begin{equation}}
\newcommand{\ee}{\end{equation}}
\newcommand{\ba}{\begin{eqnarray}}
\newcommand{\ea}{\end{eqnarray}}
\newcommand{\bas}{\begin{eqnarray*}}
\newcommand{\eas}{\end{eqnarray*}}
\newcommand{\@Bbb}[1]{\ensuremath{\mathbb #1}}

\newcommand{\B}{{\@Bbb B}}
\newcommand{\C}{{\@Bbb C}}

\newcommand{\F}{{\@Bbb F}}
\renewcommand{\P}{{\mathbb P}}
\newcommand{\bbP}{{\P}}
\newcommand{\bbE}{{\mathbb E}}
\newcommand{\Q}{{\@Bbb Q}}
\newcommand{\bQ}{{\@Bbb Q}}
\newcommand{\N}{{\@Bbb N}}

\newcommand{\bbR}{{\@Bbb R}}
\newcommand{\W}{{\@Bbb W}}

\newcommand{\bbZ}{{\@Bbb Z}}
\newcommand{\bbT}{{\@Bbb T}}

\newcommand{\si}{\sigma}

\newcommand{\Om}{\Omega}
\newcommand{\om}{\omega}

\newcommand{\@s}[1]{\ensuremath{\mathcal #1}}
\newcommand{\cA}{\@s A}
\newcommand{\cB}{\@s B}
\newcommand{\cC}{\@s C}
\newcommand{\cD}{\@s D}
\newcommand{\cE}{\@s E}
\newcommand{\cF}{\@s F}
\newcommand{\cG}{\@s G}
\newcommand{\cH}{\@s H}
\newcommand{\cI}{\@s I}
\newcommand{\cJ}{\@s J}

\newcommand{\cK}{\@s K}
\newcommand{\cL}{\@s L}
\newcommand{\cN}{\@s N}
\newcommand{\cM}{\@s M}
\newcommand{\cO}{\@s O}
\newcommand{\cP}{\@s P}
\newcommand{\cR}{\@s R}
\newcommand{\cS}{\@s S}
\newcommand{\cT}{\@s T}
\newcommand{\cV}{\@s V}
\newcommand{\cW}{\@s W}
\newcommand{\cX}{\@s X}
\newcommand{\cY}{\@s Y}
\newcommand{\cZ}{\@s Z}

\newcommand{\@bm}[1]{\ensuremath{\mathbf #1}}
\newcommand{\bma}{\@bm a}
\newcommand{\bmb}{\@bm b}
\newcommand{\bmc}{\@bm c}
\newcommand{\bmd}{\@bm d}

\newcommand{\bme}{\@bm e}
\newcommand{\bmf}{\@bm f}
\newcommand{\bmg}{\@bm g}
\newcommand{\bmh}{\@bm h}
\newcommand{\bmi}{\@bm i}
\newcommand{\bmj}{\@bm j}
\newcommand{\bmk}{\@bm k}
\newcommand{\bml}{\@bm l}
\newcommand{\bmm}{\@bm m}
\newcommand{\bmn}{\@bm n}
\newcommand{\bmo}{\@bm o}
\newcommand{\bmp}{\@bm p}
\newcommand{\bmq}{\@bm q}
\newcommand{\bmr}{\@bm r}
\newcommand{\bms}{\@bm s}
\newcommand{\bmt}{\@bm t}
\newcommand{\bmu}{\@bm u}
\newcommand{\bmw}{\@bm w}
\newcommand{\bmv}{\@bm v}
\newcommand{\bmx}{\@bm x}
\newcommand{\bx}{\@bm x}

\newcommand{\bmy}{\@bm y}
\newcommand{\bz}{\@bm z}

\newcommand{\by}{\@bm y}

\newcommand{\bmzero}{\@bm 0}

\newcommand{\ga}{\gamma}

\newcommand{\@g}[1]{\ensuremath{\mathfrak #1}}
\newcommand{\gA}{\@g A}
\newcommand{\gD}{\@g D}
\newcommand{\gJ}{\@g J}
\newcommand{\gF}{\@g F}
\newcommand{\gM}{\@g M}
\newcommand{\gR}{\@g R}

\newcommand{\commentout}[1]{{}}

\makeatother

\author{Tomasz Komorowski}
\address{Institute of Mathematics,  UMCS,
pl. Marii Curie-Sk\l odowskiej 1,
20-031, Lublin and
IMPAN,
ul. \'{S}niadeckich 8, 00-956 Warsaw, Poland, e-mail:
komorow@hektor.umcs.lublin.pl }

\author{Anna Walczuk}

\address{Institute of Mathematics,  UMCS,
pl. Marii Curie-Sk\l odowskiej 1,
20-031, Lublin and
e-mail:
walczuk.anna@gmail.com}

\thanks{Work of T. K. has been partially supported by Polish
MNiSW  grant NN 201419139.}

\begin{document}

\title[CLT for Markov processes]{Central limit theorem for Markov processes with spectral gap  in the Wasserstein metric}

\maketitle

\begin{abstract}
Suppose that  $\{X_t,\,t\ge0\}$ is  a non-stationary Markov process, taking values in a Polish metric space $E$. We prove the law of large numbers and  central limit theorem for an additive functional of the form  $\int_0^T\psi(X_s)ds$, provided that  the dual transition probability semigroup,  defined on measures, is strongly contractive in an appropriate Wasserstein metric. Function $\psi$ is assumed to be  Lipschitz on $E$.
\end{abstract}

\section{Introduction}

Suppose that $(E,\rho)$ is a Polish metric space with 
 $\cB(E)$  its Borel  $\si$-algebra and $\{X_t,\,t\ge0\}$ is a Markov process given over a certain probability space $(\Omega,{\frak F},\bbP)$. One of the fundamental problems of  classical probability theory is the question about the asymptotic behavior
of the functional $\int_0^T\psi(X_t)dt$, as $T\to+\infty$, where $\psi:E\to\bbR$ is a Borel measurable function, called an observable. One may inquire  whether the law of large numbers holds, i.e.  whether time averages
$T^{-1}\int_0^T\psi(X_t)dt$ converge in some sense to a constant, say $v_*$. If this is  the case one could further ask about the size of fluctuations around $v_*$. Typically, if the observable is not ''unusually large'', nor the process stays for a long time in the same region, properly scaled fluctuations can be described by a Gaussian random variable. This is  the contents of the central limit theorem, which states that the random variables $S_T/\sqrt{T}$, where
\begin{equation}
\label{010602}
S_T:=\int_0^T[\psi(X_s)-v_*]ds
\end{equation}
converge in law, as $T\to+\infty$, to a finite variance, centered normal random variable.

The question of the central limit theorem for an additive functional of a Markov process is a  fundamental one in classical probability theory. It can be traced back to 
the 1937 seminal article of  W. Doeblin, see  \cite{doeblin}, where the central limit theorem for
discrete time, countable Markov chains, has been shown  assuming, what is now known as  strong \emph{Doeblin's} condition. Generalizing these ideas one can prove the theorem for  more relaxed mixing conditions, such as geometric ergodicity, see e.g. Chapter 17 of \cite{meyn-tweedie}, or in the stationary setting   a spectral gap for the generator $L$ of the process in an appropriate $L^p(\mu_*)$ space, where $\mu_*$ is an invariant measure of the process, see e.g. Chapter  VI of \cite{rosenblatt}.

Starting from the 1960-s, another approach has been developed for proving central limit theorems for stationary and ergodic Markov processes, see \cite{gordin,gordin1} for the case of discrete time Markov chains and \cite{B} for continuous time Markov processes.
One uses  the solution of the Poisson equation
$-L\chi = \psi$ in $L^2(\mu_*)$ to decompose  $S_T$ into a martingale (the so called martingale approximation of $S_T$) plus a negligible term, thus
reducing the problem to a central limit theorem for martingales.  A sufficient condition for the existence of the solution to the Poisson equation is again the spectral gap of the generator. Sometimes, when $\psi$ is ''more regular'' it is quite useful to consider a smaller  (than $L^p(\mu_*)$) space, where one can prove the existence of the spectral gap, which otherwise might not exist in  the entire $L^p(\mu_*)$, see \cite{durieu,liverani}.

 In the
following decades, the martingale  approach has been developed also in the case when the Poisson equation has only an approximate solution, which converges in some sense to a generalized solution.
Using this approach it has been proved by 
Kipnis and Varadhan \cite{kipnis-varadhan} that in the case of reversible Markov processes the central limit theorem holds, provided that the  variance of $S_T/\sqrt{T}$ stays bounded, as $T\to+\infty$. The argument can be generalized also to some non-reversible processes, see e.g. \cite{varadhan} for quasi-reversible, \cite{derrennic-lin1,gordin2} for normal processes. Fairly general conditions for the central limit theorem obtained by an application of this method are formulated for discrete time, stationary Markov chains  in e.g. \cite{dedecker, maxwell-woodrofe,wu} and in \cite{holzmann} for continuous time processes.  An interesting  necessary and sufficient condition for  validity of the martingale approximation for an additive functional of a stationary Markov process of the form   \eqref{010602} can be formulated  in terms of convergence of the solutions of the corresponding resolvent equation,   see  \cite{holzmann}.

In the context of stationary Markov chains it is also worthwhile to mention the class of results where the central limit theorem (or invariance principle) is proved for a non-stationary chain starting at almost every point with respect to the stationary law - the so called quenched central limit theorem, see e.g. \cite{BI,cuny, cuny1, derrennic-lin}.   
At the end of this brief review of the existing  literature we remark  that the list of citations presented above is far from being complete.

Recently, some results have been obtained that claim   the existence of  an asymptotically stable, unique invariant measure for some classes of Markov processes, including those for which the state space 
 needs not be locally compact, see \cite{HM,HM1,Szarek,KPS10}. 
The stability we have in mind involves the convergence of the law of $X_t$ to the invariant measure  in the weak sense, as it is typical for an infinite dimensional setting.  
The Markov processes considered in the aforementioned papers  satisfy either the asymptotic strong Feller property introduced in \cite{HM}, or a somewhat weaker  e-property (see \cite{KPS10}).  In many situations they correspond to the dynamics described by a stochastically perturbed dissipative system, such as e.g. Navier-Stokes equations in two dimensions with a random forcing.

In the present article we show  the law of large numbers and central limit theorem (see Theorem \ref{lab3} below) for an additive functional of the form \eqref{010602}, with $\psi$ Lipschitz regular, for a class of Markov processes $\{X_t,\,t\ge0\}$ that besides some additional technical assumptions satisfy:  1) the strong contractive property in the Wasserstein metric for the  transfer operator semigroup associated with the process    (hypothesis H1) formulated below)
and 2) the existence of an appropriate Lyapunov function (hypothesis H3)). The technical hypotheses mentioned above include:  3) Feller property, stochastic continuity of the process (hypothesis H0)), 4) the existence of a moment of order $2+\delta$, for some $\delta>0$, for the transition probabilities (hypothesis H2)).
We stress that the processes considered  in Theorem \ref{lab3} presented below need not be stationary. In this context  the results of  \cite{guivarch,ibragimov-linnik,shirikyan} and \cite{walczuk} should be mentioned. In Theorem 19.1.1 of \cite{ibragimov-linnik} the central limit theorem is proved for every starting point  of a Markov chain that is stable in the total variation metric (this is equivalent with the uniform mixing property of the chain). In \cite{guivarch}  the theorem of this type is shown for  a     chain taking values in  a compact,  metric  state space   satisfying a stability condition that can be expressed in terms of the Wasserstein metric. The proof is conducted, via a  spectral analysis argument, applying   an analytic perturbation technique to  the   transition probability operator considered on the Banach space of bounded Lipschitz functions.  It is not clear that this kind of approach could work in our case, i.e.  for   continuous time Markov processes whose state space is allowed to be  non-compact and an observable that may be unbounded (we only require  it to be Lipschitz).  In \cite{walczuk}  Markov processes stable in the Wasserstein type metric,  stronger than the one considered here, have been examined and an analogue of   Theorem \ref{lab3}  has been shown.  After finishing this manuscript we have learned about the results of \cite{shirikyan}, where the central limit theorem for solutions of Navier-Stokes equations has been studied.

 In Section \ref{dissipative}  we apply our main result in two situations. The first application (see Section \ref{dissipative1})   concerns the  asymptotic behavior of  an additive functional    \eqref{010602}  associated with  a solution of an infinite dimensional stochastic differential equation with a dissipative drift and an additive noise (see \eqref{lab50} below), see Theorem \ref{thm012901}. Another application, presented in Section \ref{sec7.2}, is the central limit theorem  for a smooth observable of the Eulerian velocity field  that solves a  two dimensional stochastic Navier-Stokes equation (N.S.E.) system and relies on the results of \cite{HM,HM1}. It generalizes the central limit theorems for solutions of N.S.E. system forced by Gaussian white noise  that have been shown in   \cite{shirikyan}.

Finally, we describe briefly the  proof of Theorem \ref{lab3}.  The main tool we employ 
 is a suitably adapted  martingale decomposition of the additive functional in question, see Section \ref{mart-decomp}. In fact as a by-product, using the argument from Chapter 2 of  \cite{klo},  we obtain a martingale central limit theorem (see Theorem \ref{lab30b}) for a class of square integrable martingales, that  could be considered a slight generalization of  Theorem 2 of \cite{brown}. We need not assume  stationary increments, but suppose instead that the quadratic variation satisfies some form of the law of large numbers, see hypothesis M2).
The proof of this result is given in Appendix \ref{appA}. 


\section{Preliminaries and the formulation of the main result}

\label{prelim}

\subsection{Notation} Let $(E,\rho)$ be a Polish metric space 
 and let
$B(E)$, $C(E)$ and $\textrm{Lip}\, (E)$ (resp. $B_b(E)$, $C_b(E)$ and $\textrm{Lip}_b\, (E)$) be the spaces of all Borel
measurable, continuous and Lipschitz continuous (resp. bounded measurable, continuous and Lipschitz continuous) functions
on $E$, correspondingly.  
The space of all 
Lipschitz continuous 
 functions on $E$ is equipped with the pseudo-norm
$$
\|f\|_{L}:= \sup_{x\not =y} \frac{|f(x)-f(y)|}{\rho(x,y)}.
$$
 It becomes a complete norm on  $\textrm{Lip}\, (E)$  when we identify functions that differ only by a constant. Observe also that Lip$(E)$ is contained in $ C_{lin}(E)$ - the space of all continuous functions $f$, for which there exist $C>0$ and $x_0\in E$ such that $|f(x)|\le C(1+\rho_{x_0}(x))$ for all $x\in E$, where $\rho_{x_0}(x):=\rho(x,x_0)$. We shall denote by $\|f\|_{\infty,K}$ the supremum of $|f(x)|$ on a given set $K$ and omit writing the set in the notation if $K=E$.

Let $\mathcal{P}=\mathcal{P}(E)$ be the space of all Borel probability measures on  $E.$ Its subspace consisting of measures possessing the  absolute moment shall be denoted by    $\mathcal{P}_{1}=\mathcal{P}_{1}(E)$, more precisely $\nu\in\mathcal{P}_1$ iff  
$\int\rho_{x_0}d\nu <\infty$ for some (thus all) $x_{0}\in E$.

For $f \in {\rm Lip}(E)$, $x_0 \in E$ and $\nu \in{\cal  P}_1$, we have in particular
\begin{equation*}
\langle \nu,|f|\rangle \le	\|f\|_{L}\langle \nu,\rho_{x_0}\rangle + |f(x_0)| < +\infty.
\end{equation*}
Note that  $\mathcal{P}_{1}$ is  a complete metric space, when equipped with the Wasserstein metric
\begin{equation*}
d_1(\nu_1,\nu_2):= \sup_{\|f\|_L\leq 1}|\langle \nu_1,f\rangle -\langle \nu_2,f\rangle|,\quad\forall\,\nu_1,\nu_2\in\mathcal{P}_{1},
\end{equation*}
see e.g.   \cite{villani} Theorem 6.9 and Lemma 6.14. Here $\langle \nu,f\rangle:=\int f d\nu$ for any $f\in {\rm Lip}(E)$ and $\nu\in{\cal P}$,

Suppose that $\{X_t,\,t\ge0\}$ is an $E$-valued Markov process, given over a probability space $(\Omega,{\frak F},\bbP)$, whose transition probability semigroup is denoted by  $\{P^t,\,t\ge 0\}$ and initial distribution is given by a Borel probability measure $\mu_0$. Denote by $\bbE$ the expectation corresponding to $\bbP$ and by $\{{\frak F}_t,t\ge0\}$ the natural filtration of the process, i.e. the increasing  family of $\si$-algebras ${\frak F}_t:=\si(X_s,s\le t)$. We shall denote by $\mu P^t$, the dual  transition probability semigroup, describing the evolution of  the law of $X_t$. 
We have $\langle \mu,P^tf\rangle=\langle \mu P^t,f\rangle$ for all $\mu\in{\cal P}(E)$, $f\in B_b(E)$.
Particularly, $\delta_x P^t(dy)=P^t(x,dy)$ are the transition probability functions associated with the process. To abbreviate, for a given Borel probability measure $\mu$ on $E$ and a random variable $Y$, we shall write
$$
\bbE_{\mu}Y:=\int\bbE[Y|X_0=x]\mu(dx)
$$
and $\bbE_xY$ denotes the expectation corresponding to  $\mu=\delta_x$. Likewise we shall write
$\bbP_{\mu}[A]=\bbE_{\mu}1_A$ and $\bbP_{x}[A]=\bbE_{x}1_A$ for any $A\in{\cal F}$.

\subsection{Formulation of  main results}

Below, we state the list of  hypotheses we make in the present article:
\begin{itemize}
\item[H0)] the
semigroup  is \emph{ Feller}, i.e. $P^t(C_b(E))\subset C_b(E)$,
and  stochastically continuous  in the following sense: 
\begin{equation}
\label{continuous-1}
\lim_{t\to0+}P^tf(x)=f(x),\quad\forall\,x\in E,\,f\in C_b(E),
\end{equation}
\item[H1)] 
we have  $ \mu P^{t}\in \mathcal{P}_{1}$, provided that $\mu\in {\cal P}_1$. In addition,
there exist $\hat c, \gamma>0$ such that 
\begin{eqnarray}
\label{022511}
 d_{1}(\mu P^{t},\nu P^{t})\leq \hat c\, e^{-\gamma t}\,d_{1}(\mu,\nu)\label{lab},\quad\forall\,t\geq 0,\, \mu,\nu\in\mathcal{P}_{1}.
 \end{eqnarray}
\item[H2)] 
 for some (thus all) $x_{0}\in E$ there exists $\delta>0$ such that for all  $R<+\infty, $ and $T\geq 0$
\begin{equation}
\sup_{t\in[0,T]}\sup_{x\in B_R(x_0)}\int\rho^{2+\delta}_{x_0}(y)P^t\,(x,dy)<\infty ,\label{lab28}
\end{equation}
We have denoted by $B_R(x_0)$ an open ball of radius $R>0$ centered at $x_0$. 
\item[H3)] 
we assume that  $\rho^{2+\delta}_{x_0}(\cdot)$ for some $x_0$ and $\delta>0$ is  {\em a Lyapunov function} for the given process $\{X_t,\,t\ge0\}$. More specifically, we suppose that there exists $x_{0}\in E$ and $\delta>0$  such that
\begin{equation}
A_*:=\sup_{t\geq 0}\bbE\rho^{2+\delta}_{x_0}(X_t)<\infty ,\label{lab29}
\end{equation}
 \end{itemize}
 {\bf Remark 1.} Observe that condition \eqref{continuous-1} is obviously equivalent to
 \begin{equation}
\label{continuous}
\lim_{t\to0+}d_1(\delta_xP^t,\delta_x)=0,\quad\forall\,x\in
E.
\end{equation}
 {\bf Remark 2.} 
By choosing smaller of the exponents appearing in H2) and H3) we assume in what follows that the parameters $\delta$ present there are equal.

Our main result can be now formulated as follows.
\begin{thm}\label{lab3}
Suppose that   $\mu_0$ - the law of $X_0$ - belongs to ${\cal P}_1$ and an observable  $\psi \in\mbox{\em Lip}(E)$. Then,  the following are true:
\begin{enumerate}
\item[1)] \label{1} (the weak law of large numbers) if hypotheses H0) and H1)  are satisfied then, there exists a unique invariant  probability measure $\mu_*$. It
belongs to  ${\cal P}_1(E)$ and
\begin{equation}
\lim_{T\to+\infty}\frac{1}{T}\int_{0}^{T}\psi(X_{s})ds= v_{*} \label{spwl}
\end{equation}
in probability, where $v_*:=\langle \mu_*,\psi\rangle$,
\item[2)] (the existence of the asymptotic variance) if   H0) - H3) hold  then, there exists $\si\in[0,+\infty)$ such that
\begin{equation}
\lim_{T\to+\infty}\frac{1}{T}\mathbb{E}\Big[\int_{0}^{T}\tilde \psi(X_{s})ds\Big]^{2}=\sigma^2.\label{D}
\end{equation}
where $\tilde\psi(x):=\psi(x)-v_{*}$,
\item[3)] \label{2} (the central limit theorem) under the assumptions of part 2) we have
\begin{equation}
\lim_{T\to+\infty}\mathbb{P}\Big(\frac{1}{\sqrt T}\int_{0}^{T}\tilde \psi(X_{s})ds<\xi\Big)=\Phi_{\sigma}(\xi),\quad\forall\, \xi\in \mathbb{R} \label{ctg},
\end{equation}
where $\Phi_{\sigma}(\cdot)$ is the distribution function of a centered normal law with variance equal to $\si^2$.
\end{enumerate}
\end{thm}
The proofs of parts 1) and 2) of the above result are presented in Section \ref{sec4} and part 3) is shown in Section \ref{sec5}

\section{Some consequences of hypotheses  H0)-H1)} 


We start with the proof of the existence and uniqueness of the invariant probability measure claimed in part 1) of Theorem \ref{lab3}.
Uniqueness  is obvious,  in light of hypothesis H1), thus we only need to prove the existence part. Suppose that $t_{0}$ is chosen so that $\hat ce^{-\gamma t_{0}}<1.$  Using \eqref{lab} we get that $P^{t_{0}}$ is a contraction on a complete metric space $(\mathcal{P}_1(E),d_1)$. By the Banach contraction mapping principle we find   $\mu_{*}^{0}\in \mathcal{P}_1(E)$
such that $\mu_{*}^{0}P^{t_{0}}=\mu_{*}^{0}.$ Let  $\mu_{*}:= t_{0}^{-1}\int_{0}^{t_{0}}\mu_{*}^{0}P^{s}ds.$
It is easy to check that  $\mu_{*}$ is invariant under  $\{P^t,\,t\geq 0\}$. Indeed,
\begin{eqnarray*}
&&
\mu_{*}P^{t}=\frac{1}{t_{0}}\int_{0}^{t_{0}}\mu_{*}^{0}P^{s+t}ds=\frac{1}{t_{0}}\int_{t}^{t_{0}}\mu_{*}^{0}P^{s}ds
+\frac{1}{t_{0}}\int_{t_0}^{t_{0}+t}\mu_{*}^{0}P^{s}ds\\
&&=\frac{1}{t_{0}}\int_{t}^{t_{0}}\mu_{*}^{0}P^{s}ds
+\frac{1}{t_{0}}\int_{0}^{t}\mu_{*}^{0}P^{s}ds=\mu_{*}.
\end{eqnarray*}
\qed

Define $\mu Q_t:=t^{-1}\int_0^t\mu P^s ds$ for all $t>0$ and $\mu Q_N^*:=N^{-1}\sum_{n=0}^N\mu P^n$ for all integers $N\ge1$. As an easy consequence from the above  and condition H1) we obtain
\begin{prop}
\label{contract}
For any $\mu\in {\cal P}_1(E)$ we have
\begin{equation}
\label{cont}
d_1(\mu P^t,\mu_*)\le \hat ce^{-\gamma t}d_1(\mu,\mu_*),
\end{equation}
\begin{equation}
\label{aver-cont}
d_1(\mu Q_t,\mu_*)\le \frac{\hat c}{t\gamma}(1-e^{-\gamma t})d_1(\mu,\mu_*),\quad\forall\,t\ge0,
\end{equation}
and 
\begin{equation}
\label{aver-conta}
d_1(\mu Q_N^*,\mu_*)\le \frac{\hat c[1-e^{-\gamma (N+1)}]}{N(1-e^{-\gamma})}d_1(\mu,\mu_*),\quad\forall\,N\ge1.
\end{equation}
\end{prop}
\proof
Estimate \eqref{cont} is obvious. To prove \eqref{aver-cont} choose an arbitrary $\psi\in {\rm Lip}(E)$ such that $\|\psi\|_{L}\le 1$. Then
\begin{eqnarray*}
&&
|\langle \mu Q_t,\psi\rangle -\langle\mu_*,\psi\rangle|=\left|\frac{1}{t}\int_0^t\left[\langle \mu P^s,\psi\rangle -\langle\mu_* P^s,\psi\rangle\right]ds\right|\\
&&
\le \frac{\hat c}{t}d_1(\mu,\mu_*)\int_0^t e^{-\gamma s}ds 
\end{eqnarray*}
and \eqref{aver-cont} follows. The proof of \eqref{aver-conta} is analogous.
\qed

\begin{lemma}
For any  $x_0\in E$ there exists a constant $C>0$ such that
\label{bound}
\begin{equation}
\label{042712}
\sup_{t\ge0}\langle \delta_xP^t,\rho_{x_0}\rangle\le  C[\rho_{x_0}(x)+1],
\quad\forall\,x\in E,
\end{equation}
where, as we recall, $\rho_{x_0}(x):=\rho(x,x_0)$.
\end{lemma}
{\em Proof.}
From H1) we get
$$
|P^t\rho_{x_0}(x)-\langle\mu_*,\rho_{x_0}\rangle|\le d_1(\delta_x P^t,\mu_*P^t)\le \hat ce^{-\gamma t}d_1(\delta_x,\mu_*).
$$
This estimate implies \eqref{042712}
\qed

Using the above lemma and a standard truncation argument we  conclude that for any $\psi\in C_{lin}(E)$ and $t>s$
 $$
 \bbE[\psi(X_t)|{\frak F}_s]=P^{t-s}\psi(X_s),
 $$
 where $P^t\psi(x):= \langle \delta_x P^t,\psi\rangle$. 
 \begin{lemma}\label{linear}
Suppose that $\psi\in \mbox{\em Lip}(E)$. Then   $P^t\psi\in  \mbox{\em Lip}(E)$ and 
\begin{equation}
\label{052712}
\|P^t\psi\|_{L}\le \hat ce^{-\gamma t}\|\psi\|_L,\quad  \forall\,t\ge0.
\end{equation}
 Moreover, if H2) holds then $P^t(C_{lin}(E))\subset C_{lin}(E)$ for all $t\ge0$.
\end{lemma}
{\em Proof.}
From H1) we obtain that for any $x,y\in E$, $t\ge0$
\begin{eqnarray*}
&&
|P^t\psi(x)-P^t\psi(y)|\le \|\psi\|_Ld_1(\delta_x P^t,\delta_y P^t)\\
&&
\le \hat c\|\psi\|_Le^{-\gamma t}d_1(\delta_x,\delta_y)= \hat c\|\psi\|_Le^{-\gamma t}\rho(x,y)
\end{eqnarray*}
and \eqref{052712} follows.

Suppose now that $\psi\in C_{lin}(E)$. We prove first that $P^t\psi\in C(E)$. Suppose that $L>1$ and
\begin{equation}
\label{012712}
\psi_L(x):=\left\{
\begin{array}{ll}
\psi(x),&\quad\mbox{ when  }|\psi(x)|\le L,\\
 L,&\quad\mbox{ when  }\psi(x)> L,\\
- L,&\quad\mbox{ when  }\psi(x)<- L.\\
 \end{array}
 \right.
 \end{equation}
Using H2) we conclude easily that for any $R>0$, $x_0\in E$ we have
\begin{equation}
\label{022712}
\lim_{L\to+\infty}\sup_{x\in B_R(x_0)}|P^t\psi(x)-P^t\psi_L(x)|=0.
\end{equation}
From this and H0) we infer that $P^t\psi\in C(E)$.
The fact that $P^t\psi\in C_{lin}(E)$ follows directly from Lemma \ref{bound}.\qed

 \begin{lemma}\label{linear-1}
Suppose that $\psi\in  \mbox{\em Lip}(E)$ and $x\in E$. Then the   function $t\mapsto P^t\psi(x)$ is continuous for all   $t\ge0$. 
\end{lemma}
{\em Proof.}
From H1) and \eqref{052712}
$$
|P^t\psi(x)-P^s\psi(x)|\le \hat ce^{-\gamma( t\wedge s)}\|\psi\|_Ld_1(\delta_x,\delta_xP^{|t-s|}).
$$
Using \eqref{continuous} we conclude the proof of the lemma.\qed

\section{Proofs of parts 1) and 2) of Theorem \ref{lab3}} 
\label{sec4}

Some of the calculations appearing in this section are analogous to those contained in Section 3 of  \cite{walczuk} although  significant modifications are required due to the fact that we work here with a weaker metric and an observable that is allowed to be unbounded.

\subsection{Proof of part 1)} 

In case when the process is stationary (i.e. $\mu_0=\mu_*$) the result is a consequence of the continuous time version of  Birkhoff's pointwise ergodic theorem (the unique invariant measure is  then ergodic). In fact, the convergence claimed in \eqref{spwl} holds then in the almost sure sense. To prove the result in the non-stationary setting
suppose first that $\psi\in C_b(E)\cap \mbox{Lip}(E)$.
Let
 $ v(T):= \int_{0}^{T}\psi(X_{s})ds.$
It  suffices only to show that
 \begin{equation}
 \label{h11}
\lim_{T\to+\infty} \frac1T\mathbb{E} v(T)
=v_{*}\quad\mbox{and}\quad \lim_{T\to+\infty} 
 \frac{1}{T^2}\mathbb{E}v^2(T)= v_{*}^{2}.
\end{equation}
Using the Markov property we can write
\begin{eqnarray}
\label{011712}
&&\frac{1}{T}\mathbb{E}v(T)=\frac{1}{T}\int_{0}^{T}\mathbb{E}\psi(X_{s})ds\\
&&
=\frac{1}{T}\int_{0}^{T}
\langle\mu_{0}P^{s},\psi\rangle ds \xrightarrow {T\to \infty}\langle \mu _{*},\psi\rangle =v_{*}.\nonumber
\end{eqnarray}
On the other hand
\begin{eqnarray}
&&\frac{1}{T^2}\mathbb{E}v^2(T)=
\frac{1}{T^{2}}\,\mathbb{E}\Big(\int_{0}^{T}\!\!\psi(X_{t})dt\int_{0}^{T}\!\!\psi(X_{s})ds\Big) \nonumber\\
&&=\frac{2}{T^{2}}\!\!\int_{0}^{T}\!\!\int_{0}^{t}\!\,\mathbb{E}[\psi(X_{t})\psi(X_{s})]dt ds.\label{lab11}
\end{eqnarray}
The right hand side of   \eqref{lab11} equals
 \begin{eqnarray*}
&& \frac{2}{T^{2}}\int_{0}^{T}\!\!\int_{0}^{t}\!\,\mathbb{E}\big[\psi(X_{s})P^{t-s}\psi(X_{s})\big] dtds\\
 &&=\frac{2}{T^{2}}\int_{0}^{T}\!\!\int_{0}^{t}\! \langle\mu_0P^{s},\psi P^{t-s}\psi\rangle dt ds.
 \end{eqnarray*}
We claim that for any $\varepsilon>0$ there exists $T_0$ such that for all $T\ge T_0$ we have
\begin{equation}
\Big|\frac{2}{T^{2}}\int_{0}^{T}\!\!\int_{0}^{t}\!\langle \mu_0 P^{s},\psi (P^{t-s}\psi-\upsilon_{*})\rangle dtds\Big|<\varepsilon.
\label{lab4}
\end{equation}
Accepting this claim (proved  below) for a moment we conclude that 
 \begin{eqnarray*}
&&\lim\limits_{T\rightarrow \infty}\mathbb{E}\Big[\frac{v(T)}{T}\Big]^{2}=\lim\limits_{T\rightarrow \infty}\frac{2}{T^{2}}\upsilon_{*}\int_{0}^{T}\!\!t\, dt\left[
\frac{1}{t}\int_{0}^{t}\langle\mu_{0}P^{s},\psi\rangle d s \right]
=\upsilon_{*}^{2}.
 \end{eqnarray*}
The last equality follows from \eqref{011712}.

\subsection*{Proof of (\ref{lab4})}

We shall need the following two lemmas:
\begin{lemma}
\label{lma2}
Suppose that $\psi\in {\rm Lip}(E)\cap C_b(E)$. Then, for any $\varepsilon >0$ and a compact subset $K\subset E$ there exists $T_{0}$ such that for any  $T\geq T_{0}$
\begin{equation}\sup \limits _{x\in
K}\Big|\frac{1}{T}\int_{0}^{T}P^{s}\psi(x)ds-\upsilon_{*}\Big|<\varepsilon.\label{lab6}\end{equation}
\end{lemma}
{\em Proof.}
Note that $\{P^{s}\psi, {s\geq 0}\}$ forms  an equicontinuous and uniformly bounded family of functions. Indeed, from condition  \eqref{lab} we have
 \begin{eqnarray*}
&&\Big|P^{s}\psi(x_{1})-P^{s}\psi(x_{2})\Big|=\Big|\langle \delta_{x_{1}}P^{s},\psi\rangle-\langle \delta_{x_{2}}P^{s},\psi\rangle\Big|\\
&&\leq
d_{1}(\delta_{x_{1}}P^{s},\delta_{x_{2}}P^{s})\|\psi\|_{L}\leq \hat c\,e^{-\gamma s}d_{1}(\delta_{x_{1}},\delta_{x_{2}})\|\psi\|_{L}\\
&&
\leq
\hat c\,e^{-\gamma s}\rho(x_{1},x_{2})\|\psi\|_{L}.
 \end{eqnarray*}
 for  all $x_1,x_2$, $s\ge0$. A uniform bound on the family is provided by $\|\psi\|_{\infty}$.
On the other hand,
 $$
 \psi_T(x):=\frac{1}{T}\int_{0}^{T}P^{s}\psi(x)ds,\quad T\ge 1
 $$ 
  is equicontinuous and uniformly bounded, so from the Arzela-Ascoli theorem, see  Theorem IV.6.7 of \cite{d-s}, we conclude that it is compact in the uniform topology on compact sets, as $T\to+\infty$. The lemma is a consequence of \eqref{011712} applied for $\mu_0=\delta_x$.
\qed

\begin{lemma}
\label{lma1}
For any $\varepsilon >0$ there exists a compact set $K$ and $T_0>0$ such that 
\begin{equation}\frac{1}{T}\int_{0}^{T}\!\!\mu_{0} P^{t}(K^{c})\,dt<\varepsilon, \quad \forall\,   T\geq
T_0.\label{lab10}\end{equation}
\end{lemma}
{\em Proof.}
Condition H1) implies tightness of $\mu_{0}P^{t}$ as $t\to+\infty$. This of course implies tightness of the ergodic averages.
\qed

Choose an arbitrary $\varepsilon>0$,  compact set $K$ and $T_0$ as in Lemma \ref{lma1}. Then find $T_0^*$ as in Lemma \ref{lma2} for given $\varepsilon>0$ and compact set $K$. The left hand side of   \eqref{lab4} can be estimated by 
\begin{eqnarray*}
&&
\Big|\frac{2}{T^{2}}\int_{0}^{T}\!\!\int_{0}^{t}\!\langle \mu_0 P^{s},1_{K}\psi (P^{t-s}\psi-\upsilon_{*})\rangle dtds\Big|
\\
&&+\Big|\frac{2}{T^{2}}\int_{0}^{T}\!\!\int_{0}^{t}\!\langle \mu_0 P^{s},1_{K^c}\psi (P^{t-s}\psi-\upsilon_{*})\rangle dtds\Big|.
\end{eqnarray*}
Denote the terms of the above sum by $I_T$ and $I\!I_T$ respectively.
Since  $P^{s}$ is contractive on $B_b(E)$ from \eqref{lab6} we conclude
\begin{eqnarray*}
&&I_T=
\Big|\frac{2}{T^{2}}\int_{0}^{T}\!(T-s)\!\!\left\langle \mu_0 P^{s},\left[\frac{1}{T-s}\int_{0}^{T-s} (P^{t}\psi-\upsilon_{*})dt\right]\psi1_{K}\right\rangle ds\Big|
\end{eqnarray*}
changing variables $s:=T-s$ we can write
\begin{eqnarray*}
&&I_T=
\Big|\frac{2}{T^{2}}\int_{0}^{T_0^*}\!s\left\langle \mu_0 P^{s},\left[\frac{1}{s}\int_{0}^{s} (P^{t}\psi-\upsilon_{*})dt\right]\psi1_{K}\right\rangle ds\Big|\\
&&
+\Big|\frac{2}{T^{2}}\int_{T_0^*}^{T}\!s\left\langle \mu_0 P^{s},\left[\frac{1}{s}\int_{0}^{s} (P^{t}\psi-\upsilon_{*})dt\right]\psi1_{K}\right\rangle ds\Big|\\
&&
\le 2\|\psi\|_{\infty}^2\left(\frac{T_0^*}{T}\right)^2+
\frac{2\varepsilon}{T^{2}}\|\psi\|_{\infty}\int_{0}^{T}s\,ds=2\|\psi\|_{\infty}^2\left(\frac{T_0}{T}\right)^2+\varepsilon \|\psi\|_{\infty}.
\end{eqnarray*}
Hence 
$$
\limsup_{T\to+\infty}I_T\le \varepsilon \|\psi\|_{\infty}.
$$
On the other hand, from  \eqref{lab10} we conclude that 
\begin{eqnarray*}
&&I\!I_T\leq
\frac{2\|\psi\|_{\infty}^{2}}{T^{2}}\int_{0}^{T}tdt\,\left[\frac{1}{t}\int_{0}^{t}\mu_{0} P^{s}(K^{c})\,ds\right].
\end{eqnarray*}
Using Lemma \ref{lma1} we obtain that 
\begin{eqnarray*}
&&I\!I_T\leq  2\|\psi\|_{\infty}^2\left(\frac{T_0}{T}\right)^2+\varepsilon\|\psi\|_{\infty}^{2}.
\end{eqnarray*}
Thus also 
$$
\limsup_{T\to+\infty}I\!I_T\le \varepsilon \|\psi\|_{\infty}^2.
$$
Since $\varepsilon>0$ can be arbitrary we conclude \eqref{lab4}, thus obtaining \eqref{spwl} for $\psi$ Lipschitz and bounded.

Now we remove the restriction of  boundedness of the observable $\psi$. Let $L>1$ be arbitrary.
 Recall that $\psi_L$ is given by \eqref{012712}. Using the already proven part of the theorem
 we get
 $$
\lim_{T\to+\infty}\frac{1}{T}\int_{0}^{T}\psi_L(X_{s})ds= v_{*}^{(L)}=\langle \mu_*,\psi_L\rangle.
 $$
Let $\psi^{(L)}:=|\psi-\psi_L|$. Since $\mu_*\in{\cal P}_1(E)$ we have $\langle\mu_*,|\psi|\rangle<+\infty$. It is clear therefore that
 \begin{equation}
 \label{xxx-2}
\lim_{L\to+\infty} \langle \mu_*,\psi^{(L)}\rangle=0.
 \end{equation}
 \begin{lemma}
 \label{lm-x1} 
 We have
 $$
\lim_{L\to+\infty}\limsup_{t\to+\infty} \bbE\psi^{(L)}(X_t)=0.
 $$
 \end{lemma}
 \proof
By virtue of assumption H1) we can write
 \begin{eqnarray*}
&&
 \left|\bbE\psi^{(L)}(X_t)-\langle\mu_*,\psi^{(L)}\rangle\right|= \left|\langle\mu_0 P^t,\psi^{(L)}\rangle-\langle\mu_*P^t,\psi^{(L)}\rangle\right|\\
 &&
 \le 
 \hat ce^{-\ga t}d_1(\mu_0,\mu_*)\|\psi\|_{L}.
 \end{eqnarray*}
 The conclusion of the lemma follows then from the above and \eqref{xxx-2}.
 \qed
 
 From  the above lemma we conclude easily that
  $$
\lim_{L\to+\infty}\limsup_{T\to+\infty}\bbE\left|\frac{1}{T}\int_{0}^{T}[\psi(X_s)-\psi_L(X_s)]ds\right|=0,
 $$
which, thanks to \eqref{xxx-2},  yields \eqref{spwl}.

 \subsection{Corrector and its properties.}

With no loss of generality we may and shall assume that $v_{*}\coloneqq \langle\mu_{*},\psi\rangle=0,$ otherwise we would  consider $\psi\coloneqq\!\!\psi\!-v_{*}.$

\begin{lemma}\label{lab21}
Suppose that $\psi\in \mbox{\em Lip}(E)$. The functions 
\begin{equation}
\label{chi-t}
\chi_{t}\coloneqq\int_{0}^{t}P^{s}\psi ds
\end{equation}
  converge uniformly on bounded sets, as $t\rightarrow\infty.$
\end{lemma}
{\em Proof.}
We show that   $\{\chi_{t}, t\geq 1\}$ satisfies Cauchy's condition on bounded subsets of $E$, as $t\to+\infty$.
Since $\langle \mu_{*}P^{s},\psi\rangle=0$ for all $s\ge 0$, for any  $u>t$ we have
\begin{eqnarray}
\label{lab18}
\Big|\int_{0}^{u}P^{s}\psi(x)\, ds -\int_{0}^{t}P^{s}\psi(x)\,ds\Big|
\leq\int_{t}^{u}\Big|\langle \delta_xP^{s},\psi\rangle-\langle \mu_{*}P^{s},\psi\rangle\Big|\, ds. 
\end{eqnarray}
Suppose that $\varepsilon>0$ is arbitrary. Using the definition of the metric $d_1$, the  right hand side of  \eqref{lab18} can be estimated by
\begin{eqnarray*}
&&\int_{t}^{u}\|\psi\|_{L} d_{1}(\delta_{x}P^{s},\mu_{*}P^{s})\,ds\leq
\hat c\|\psi\|_{L}d_{1}(\delta_{x},\mu_{*})\int_{t}^{u}\, e^{-\gamma s}\,ds\\
&&\leq \hat c\,\|\psi\|_{L} e^{-\gamma t}d_{1}(\delta_{x},\mu_{*})<\varepsilon,
\end{eqnarray*}
 provided that $u>t\ge t_0$ and $t_0$ is sufficiently large.$\Box$

The limit  
\begin{equation}
\label{chi}
\chi:=\lim_{t\to+\infty}\chi_{t}=\int_{0}^{\infty}P^{s}\psi\,ds
\end{equation} is called a {\em corrector}.

{\bf Remark.} This object is sometimes also referred to as the {\em potential}, as it formally solves the Poisson equation $-L\chi=\psi$, where $L$ is the generator of the semigroup $\{P^t,\,t\ge0\}$. We shall not use  this  equation explicitly in our paper, since we have not made an assumption that the semigroup is strongly continuous on the space of Lipschitz functions, so the generator is not defined in our case.

\begin{lemma}\label{lab38}
We have $\chi\in\mbox{\em Lip}(E)$.
In addition for any $T>s$
 \begin{eqnarray}
 \label{042312}
\mathbb{E}[\chi(X_{T})|\mathfrak{F}_{s}]=
\lim_{t\rightarrow +\infty}\mathbb{E}[\chi_{t}(X_{T})|\mathfrak{F}_{s}].
\end{eqnarray}
\end{lemma}
{\em Proof.}
Note that
\begin{eqnarray}
&&|\chi_{t}(x)-\chi_{t}(y)|= \left|\int_{0}^{t}P^{s}\psi(x)ds-\int_{0}^{t}P^{s}\psi(y)ds\right|\nonumber\\
&&\leq \int_{0}^{t}\left|\int\psi(z)\delta_{x}P^{s}(dz)-\int\psi(z)\delta_{y}P^{s}(dz)\right|ds.\label{lab19}
\end{eqnarray}
Similarly as in the proof of  Lemma \ref{lab21}, the right hand side of  \eqref{lab19} can be estimated by
\begin{eqnarray}
\label{012312}
&&\|\psi\|_{L} \int_{0}^{t}d_{1}(\delta_{x}P^{s},\delta_{y}P^{s})ds
\leq \hat c\|\psi\|_{L} d_{1}(\delta_{x},\delta_{y})\int_{0}^{t}e^{-\gamma s}ds\\
&&\leq C\|\psi\|_{L}\rho(x,y)(1-e^{-\gamma t})\nonumber
\end{eqnarray}
for some  $C>0$ independent of $t,x,y$.
Letting $t\to+\infty$ we get the first part of the lemma. 

Let us fix $x_0\in E$. From Lemma \ref{lab21} and \eqref{012312} it follows that there exists $C>0$ such that
\begin{eqnarray}
\label{022312}
&&|\chi_{t}(x)|\le C[1+\rho_{x_0}(x)],\quad\forall\,t>0,\,x\in E.
\end{eqnarray}
From H1)  and Lebesgue dominated convergence theorem  it follows that
\begin{eqnarray}
\label{032312}
\lim_{t\to+\infty}P^{T-s}\chi_t(x)= P^{T-s}\chi(x)\quad\forall\,x\in E.
\end{eqnarray}
Hence,
\begin{eqnarray*}
&&
\lim_{t\to+\infty}\mathbb{E}[\chi_t(X_{T})|\mathfrak{F}_{s}]=\lim_{t\to+\infty} P^{T-s}\chi_t(X_s) \\
&&
=P^{T-s}\chi(X_s)=\mathbb{E}[\chi(X_{T})|\mathfrak{F}_{s}]
\end{eqnarray*}
and \eqref{042312} follows.
$\Box$

\subsection{Proof of part 2)}
After a simple calculation we get
\begin{eqnarray*}
&&\frac{1}{T}\mathbb{E} \Big(\int_{0}^{T}\psi\,ds\Big)^{2}=\frac{2}{T}\int_{0}^{T}\!\! \left\langle \mu_0P^s, \psi \int_{0}^{T-s}P^{t}\psi\,dt\right\rangle ds.
\end{eqnarray*}
Note that integrals appearing on both sides of the above equality make sense in light of assumption H3) and the fact that $\psi \in C_{lin}(E)$.
Denoting the right hand side by $E(T)$ we can write that
\begin{eqnarray}
\label{012401}
&&\Big|E(T)-\frac{2}{T}\int_{0}^{T}\! \!\! \left\langle \mu_0 P^{s} ,\psi \chi\right\rangle ds\Big|\\
&&=\frac{2}{T}\left|\int_{0}^{T}\!\! \left\langle \mu_0P^s, \psi(\chi-\chi_{T-s})\right\rangle ds\right|,\nonumber
\end{eqnarray}
see \eqref{chi-t} and \eqref{chi} for the definitions of $\chi_t$ and $\chi$ respectively.
Using \eqref{lab19} we conclude that there exist $C>0$ and $x_0\in E$ such that
\begin{equation}
\label{052401}
\left|\psi(x)\chi(x)\right|+\left|\psi(x)\chi_u(x)\right|\le C\rho^2_{x_0}(x),\quad\forall \,x\in E, u>0.
\end{equation}
Choose an arbitrary $\epsilon>0$.
According to H3) we can find a sufficiently large $R>0$ such that  
\begin{eqnarray}
\label{022401}
&&\left| \left\langle \mu_0P^s, \psi(\chi-\chi_{T-s})1_{B_R^c(x_0)}\right\rangle\right|
\\
&&\le C\bbE\left[\rho^2(X_s,x_0),\rho_{x_0}(X_s)\ge R\right]<
 \frac{\epsilon}{2},\quad\forall\,0\le s\le T.\nonumber
\end{eqnarray}
On the other hand from Lemma \ref{lab21} we can choose $M>0$ large enough so that
\begin{equation}
\label{032401}
\left| \left\langle \mu_0P^s, \psi(\chi-\chi_{T-s})1_{B_R(x_0)}\right\rangle\right|\le
 \frac{\epsilon}{2},\quad\forall\, T-s>M.
\end{equation}
Combining \eqref{022401} with \eqref{032401} we  conclude that the right hand side of \eqref{012401} converges to $0$, as $T\to+\infty$. Since $ \mu_0P^s$ tends to $ \mu_*$, as $s\to+\infty$,  weakly in the sense of convergence of measures, we conclude from H3) and \eqref{052401} that  $\langle \mu_*,|\psi\chi|\rangle<+\infty$
and
\begin{eqnarray*}
&&\lim\limits_{T\to \infty}\frac{2}{T}\int_{0}^{T}\! \!\! \left\langle \mu_0 P^{s} ,\psi \chi\right\rangle ds=2\langle \mu_*,\psi\chi\rangle.
\end{eqnarray*}

\section{Proof of part 3) of Theorem \ref{lab3}}

 \label{sec5}

 \subsection{A central limit theorem for martingales}
Suppose that $\{{\frak F}_n,\,n\ge0\}$ is a filtration over $(\Omega,{\frak F},\bbP)$ such that ${\frak F}_0$ is trivial and $\{Z_n,\,n\ge1\}$ is a sequence of square integrable martingale differences, i.e. it is  $\{{\frak F}_n,\,n\ge1\}$  adapted, $\bbE Z_n^2<+\infty$  and $\bbE[Z_n\big|{\frak F}_{n-1}]=0$ for all $n\ge1$. Define also the martingale 
$$
M_N:=\sum_{j=1}^NZ_j, \quad N\ge1,\quad M_0:=0.
$$
  Its quadratic variation equals $\langle M\rangle_{N}:=\sum_{j=1}^N\bbE\left[Z_j^2\big|{\frak F}_{j-1}\right]$ for $N\ge1$. 
Assume also that:
\begin{itemize}
\item[M1)] \label{lab310} for every $\varepsilon >0,\,$\, 
$$
\lim_{N\to+\infty}\frac {1}{N}\sum_{j=0}^{N-1} \mathbb{E}\Big[ Z_{j+1}^2, \,
|Z_{j+1} | \ge \varepsilon\sqrt{N}  \Big]=0,
$$  
\item[M2)] we have
 \begin{equation}
 \label{sublinear}
 \sup_{n\ge1}\bbE Z_n^2<+\infty
 \end{equation} and 
 there exists $\sigma\ge0$ such that
    $$\lim_{K\rightarrow\infty}\limsup_{\ell\rightarrow\infty}\frac{1}{\ell}\sum_{m=1}^{\ell}\mathbb{E}\Big|\frac{1}{K}
    \bbE\left[\langle M\rangle_{mK}-\langle M\rangle_{(m-1)K}\Big|{\frak F}_{(m-1)K}\right] -\sigma^{2}\Big|=0$$
    and 
\item[M3)] for every $\varepsilon >0$
\begin{equation}
\label{m3}
\lim_{K\rightarrow\infty}\limsup_{\ell\rightarrow\infty}\frac{1}{\ell K}\sum_{m=1}^{\ell}\sum_{j=(m-1)K}^{mK-1}\mathbb{E}[1+Z_{j+1}^2,\, |M_{j}-M_{(m-1)K}|\geq\varepsilon\sqrt{\ell K}]=0.
\end{equation}
\end{itemize}
\begin{thm}\label{lab30b}
Under the assumptions made above we have 
\begin{equation}
\label{M2b}
\lim_{N\to+\infty}\frac{\mathbb{E}\langle M\rangle_{N}}{N}
    =\sigma^{2}
\end{equation} and
\begin{equation}
\label{052501}
\lim_{N\to\infty}  \, \mathbb{E}  e^{i \theta
  M_N/\sqrt N} = e^{-\sigma^2 \theta^2/2},\quad\forall\,\theta\in\bbR.
\end{equation}
\end{thm}
The proof of this theorem is a modification of the argument contained in Chapter 2 of \cite{klo}. In order not to divert  reader's attention we postpone its presentation  till Appendix \ref{appA}.

\subsection{Martingale approximation and the proof of the central limit theorem}

We use the martingale technique of proving the central limit theorem for an additive functional of a  Markov process  and represent $\int_{0}^{T}\psi(X_{s})ds$ as a sum of a martingale and a ''small''  remainder term that vanishes, after dividing by $\sqrt{T}$, as $T\rightarrow\infty$. The theorem is then a consequence of an appropriate  central limit theorem for martingales, see Theorem \ref{lab30b} modeled after a theorem presented  in Section 2.1 of \cite{klo}. The proof of this result is presented in Appendix \ref{appA}.

\subsubsection{Reduction to the central limit theorem for  martingales}
\label{mart-decomp}
Note that
\begin{eqnarray}
\label{decomp}
&&\frac{1}{\sqrt{T}}\int_{0}^{T}\psi(X_s)\,ds
=\frac{1}{\sqrt{T}}M_{T}+R_{T}
\end{eqnarray}
where
\begin{eqnarray}
\label{mart}
M_{T}\coloneqq\chi(X_{T})-\chi(X_{0})+\int_{0}^{T}\psi(X_{s})\,ds.
\end{eqnarray}
and
$$R_{T}\coloneqq \frac{1}{\sqrt{T}}\left[\chi(X_{0})-\chi(X_{T})\right].$$

\begin{prop}\label{lab20}
Under the assumptions of Theorem $\ref{lab3}$ the process
$\{
M_{T},\,T\ge0\}
$ is a martingale with respect to the  filtration  $\{\mathfrak{F}_T, T\ge 0\}.$
\end{prop}
{\em Proof.}
From H3) it follows that $\mathbb{E}|M_{T}|<\infty.$ We have
\begin{eqnarray*}
&&\mathbb{E}[M_{T}|\mathfrak{F}_{s}]
=\mathbb{E}[\chi(X_{T})|\mathfrak{F}_{s}]-\chi(X_{0})+
\int_{0}^{s}\mathbb{E}[\psi (X_{u})|\mathfrak{F}_{s}]du\\
&&+\int_{s}^{T}\mathbb{E}[\psi (X_{u})|\mathfrak{F}_{s}]du.
\end{eqnarray*}
The last term on the right hand side equals
\begin{eqnarray*}
&&
\int_{s}^{+\infty}P^{u-s}\psi (X_{s})du-\int_{T}^{+\infty}P^{u-T}(P^{T-s}\psi) (X_{s})du
\\
&&
=
\chi(X_{s})-\mathbb{E}[\chi(X_{T})|\mathfrak{F}_{s}].
\end{eqnarray*}
This  ends the proof of the martingale property. 
\qed

\begin{lemma}\label{lab12}
The random variables $R_{T}$ converge to $0$, as $T\to+\infty$, in the $L^{1}$-sense.
\end{lemma}
{\em Proof.}
Since $\mathbb{E}|\chi(X_{0})|<+\infty$ we conclude that
\begin{eqnarray*}
&&\frac{1}{\sqrt{T}}\mathbb{E}|\chi(X_{0})|\xrightarrow {T\to \infty}0
\end{eqnarray*}
On the other hand
\begin{eqnarray}
&&\frac{1}{\sqrt{T}}\mathbb{E}|\chi(X_{T})|
=\frac{1}{\sqrt{T}}\langle\mu_0 P^{T},|\chi|\rangle. \label{lab36}
\end{eqnarray}
Since  $\mu_{*}P^{T}=\mu_{*}$ we can rewrite the right hand side of  \eqref{lab36} as being equal to
\begin{eqnarray}
&&\frac{1}{\sqrt{T}}\big[\langle\mu_0 P^{T},|\chi|\rangle\!-\!\!\langle\mu_* P^{T},|\chi|\rangle\big]
\!+\!\frac{1}{\sqrt{T}}\|\chi\|_{L^1(\mu_{*})}\label{lab37}\\
&&\le
\frac{1}{\sqrt{T}}\|\chi\|_{L}d_{1}(\mu_{0} P^{T},\mu_{*}P^{T})+\frac{1}{\sqrt{T}}\|\chi\|_{L^1(\mu_{*})}\nonumber\\&&\stackrel{\eqref{lab}}{\leq}
\frac{\hat c}{\sqrt{T}}\|\chi\|_{L} e^{-\gamma T} d_{1}(\mu_{0} ,\mu_{*})+\frac{1}{\sqrt{T}}\|\chi\|_{L^1(\mu_{*})}\xrightarrow {T\to \infty}0.\nonumber\Box
\end{eqnarray}

\subsubsection{Verification of the assumptions of Theorem \ref{lab30b}}

We assume that all the constants appearing in  ensuing estimations and designated by the letter $C$ are strictly positive and do not depend on $N, K, \ell.$

 We verify the assumptions of Theorem \ref{lab30b}   for the martingale defined in \eqref{mart} and $Z_n:=M_n-M_{n-1}$ for $n
\ge1$.
Then, part 3) of Theorem \ref{lab3} follows thanks to decomposition \eqref{decomp}, Lemma \ref{lab12} and the fact that for any $\varepsilon>0$
\begin{equation}
\label{vanish}
\lim_{N\rightarrow\infty}\mathbb{P}[\sup_{T\in[N,N+1)}|M_T/\sqrt{T}-M_N/\sqrt{N}|\geq\varepsilon]=0.
\end{equation}
To see equality \eqref{vanish} note that  the probability under the limit is less than or equal to  
\begin{eqnarray*}
&&\mathbb{P}[\sup_{T\in[N,N+1)}|M_T-M_N|\ge \varepsilon\sqrt{N}/2]\\
&&
+\mathbb{P}\left[|M_N|\left[\frac{1}{N^{1/2}}-\frac{1}{(N+1)^{1/2}}\right]\ge \varepsilon/2\right]\\
&&
\le \frac{C}{N\varepsilon^2}\bbE[\langle M\rangle_{N+1}-\langle M\rangle_{N}]+\frac{C}{N^3\varepsilon^2}\bbE[\langle M\rangle_{N}].
\end{eqnarray*}
The last inequality follows from Doob and Chebyshev estimates and an elementary inequality $N^{-1/2}-(N+1)^{-1/2}\le CN^{-3/2}$ that holds for all $N\ge1$ and some constant $C>0$. The first term on the right hand side vanishes as $N\to+\infty$, thanks to \eqref{M2b}, while the second is clearly smaller than
$$
\frac{C}{N^3\varepsilon^2}\sum_{k=1}^N\bbE Z_k^2\to 0
$$
as $N\to+\infty$, thanks to \eqref{sublinear}.

\subsubsection*{Condition M1)}
We recall the shorthand notation  
$$
\mu_{0}Q_{N}^* \coloneqq \frac{1}{N} \sum_{n=1}^{N}\mu_{0} P^{n-1}.
$$ Note that, by the Markov property
\begin{eqnarray}
\label{010401}
&&\frac {1}{N}\sum_{n=1}^{N} \mathbb{E}\Big[ Z_{n}^2, \, | Z_{n} | \ge \varepsilon\sqrt{N}  \Big]=\langle\mu_{0}Q_{N}^*,G_{N}\rangle,
\end{eqnarray}
where $G_{N}(x)\coloneqq \mathbb{E}_x\Big[ Z_{1}^2, \,| Z_{1} | \ge \varepsilon\sqrt{N}\Big].$ We claim that the right hand side of \eqref{010401} vanishes, as $N\to+\infty$. The proof shall be based on  the following.
\begin{lemma}
\label{convergence}
Suppose that
$\{\mu_{N},\,N\ge1\}\subset {\cal P}$ weakly converges to $\mu$,
 $\{G_{N},\,N\ge1\}\subset B(E)$ converges  to $0$ uniformly on compact sets and  there exists $\delta>0$ such that
\begin{equation}
G_*:=\limsup_{N\rightarrow \infty}\langle\mu_N, |G_{N}|^{1+\delta}\rangle<+\infty.\label{lab30a3}
\end{equation}
Then,
$
\lim_{N\to+\infty}\langle\mu_N,G_{N}\rangle= 0.
$
\end{lemma}
\proof
Suppose that $K$ is compact and $\varepsilon>0$ is arbitrary. Then,
\begin{equation}\label{lab30}
|\langle\mu_N,G_{N}\rangle|\le |\langle\mu_N,G_{N}1_{K}\rangle|+|\langle\mu_N,G_{N}1_{K^c}\rangle|
\end{equation}
Using H\"older's inequality and choosing appropriately the compact set we can estimate the second term by
\begin{equation*}
 \langle\mu_N,|G_{N}|^{1+\delta}\rangle^{1/(1+\delta)}\mu_{N}^{\delta/(1+\delta)}(K^c)
\leq G_*\mu_{N}^{\delta/(1+\delta)}(K^c)\leq \varepsilon ,\quad\forall\,N\ge1.
\end{equation*}
The first term can be estimated by $\|G_N\|_{\infty,K}\to0$, as $N\to+\infty$. Since $\varepsilon>0$ has been chosen arbitrarily the conclusion of the lemma follows.
\endproof

From Proposition \ref{contract} we have $\lim_{N\to+\infty}\mu_{0}Q_{N}^*=\mu_{*}$, weakly.
To prove that $\{G_N,\,N\ge1\}$ converges to $0$ uniformly on compact sets it suffices to show that for any $x_0\in E$ and $\varepsilon, R>0$ we have
\begin{equation}
\label{conv-mart}
\lim_{N\to+\infty}\sup_{x\in B_R(x_0)}\mathbb{E}_x\left[M^{2}_{1},\,| M_{1}|\geq\varepsilon\sqrt{N}\right]= 0.
\end{equation}
To show  \eqref{conv-mart} it suffices only to prove that for any $x_0\in E$ and $R>0$ there exists $\delta>0$ such that
\begin{equation}
\label{conv-martb}
M_R^*:=\limsup_{N\to+\infty}\sup_{x\in B_R(x_0)}\bbE_x|M_{1}|^{2+\delta}<+\infty.
\end{equation}
Equality \eqref{conv-mart} then follows from the above and Chebyshev's inequality. 
Using definition \eqref{mart} we get, with $\delta$ as in the statement of H2), that
\begin{eqnarray}
\label{010501}
&&
\mathbb{E}_x |M_{1}|^{2+\delta}\le C\left\{\bbE_x\left[\left|\chi(X_{1})-\chi(X_{0})\right|^{2+\delta}\right]\vphantom{\int_0^1}\right.\nonumber
\\
&&
+\left.\bbE_x\left[\int_{0}^{1}\left|\psi(X_{s})\right|^{2+\delta}\,ds\right]\right\}
\leq C\left\{(\|\chi\|_{L}+1)^{2+\delta}\langle \delta_xP^1,\rho^{2+\delta}_{x}\rangle\right.\nonumber\\
&&
\left.+(\|\psi\|_{L}+|\psi(x_0)|+1)^{2+\delta}\langle \delta_x Q_1,\rho^{2+\delta}_{x_0}+1\rangle\right\}.
\end{eqnarray}
Thus, \eqref{conv-martb} and also \eqref{conv-mart} follow. In particular \eqref{010501}  implies that $G_N(x)$ converges to $0$ uniformly on compact sets. 
On the other hand, since
$$
|G_N(x)|^{1+\delta/2}\le \mathbb{E}_x |M_{1}|^{2+\delta},\quad\forall\,x\in E
$$
condition  \eqref{lab30a3} easily follows from \eqref{010501} and  hypothesis H3).
This concludes the proof of M1).

\subsubsection*{Condition M2)} Note that
\begin{eqnarray}
\label{012611}
&&
\mathbb{E}Z^{2}_{n}\le 2\left\{\bbE\left|\chi(X_{n+1})-\chi(X_{n})\right|^{2}+\int_{n}^{n+1}\bbE|\psi(X_s)|^2ds\right\}\nonumber\\
&&
\leq C\left\{(\|\chi\|_{L}+1)^{2}[\bbE\rho^{2}_{x_0}(X_{n+1})+\bbE\rho^{2}_{x_0}(X_{n})]\vphantom{\int_0^1}\right.\nonumber\\
&&
\left.+[\|\psi\|_{L}+|\psi(x_0)|+1]^{2}\int_n^{n+1}[\bbE \rho^{2}_{x_0}(X_s)+1]ds\right\}
\end{eqnarray}
and thanks to H3) we have $\sup_{n\ge0}\mathbb{E}Z^{2}_{n}<+\infty$ so \eqref{sublinear} holds.

 Using the Markov property we can write for any $\si\ge0$ (to be specified later) 
 \begin{eqnarray*}
&&\frac{1}{\ell}\sum_{m=1}^{\ell}\mathbb{E}\left|
    \bbE\left\{\frac{1}{K}\left[<\!M\!>_{mK-1}-<\!M\!>_{(m-1)K}\right] -\sigma^{2}\Big|{\frak F}_{(m-1)K}\right\}\right|\\
&&=\frac{1}{\ell}\sum_{m=0}^{\ell-1} \langle\mu_{0} P^{(m-1)K},|H_K|\rangle.
\end{eqnarray*}
with
$$
H_K(x)\coloneqq \mathbb{E}_x\left[\frac{1}{K}<\!M\!>_{K}-\sigma^{2}\right]= \mathbb{E}_x\left[\frac{1}{K}M^2_{K}-\sigma^{2}\right].
$$
Note that
\begin{equation}
\label{011506}
H_K(x)=\frac{1}{K}\sum_{j=0}^{K-1}P^jJ(x),
\end{equation}
where
$J(x):=\bbE_x\langle M\rangle_1-\sigma^2$.
Let 
$
\mu_{0}Q_{\ell}^K\coloneqq 1/\ell\,\sum_{m=1}^{\ell}\mu_{0} P^{(m-1)K}.
$
\begin{lemma}
\label{cont-hk}
For any $K\ge 1$ we have $H_K\in C(E)$. Moreover, for $\delta>0$ as in hypothesis H3) we have
\begin{equation}
\label{square-est}
 \limsup_{\ell\to+\infty}\langle\mu_{0}Q_{\ell}^K,|H_K|^{1+\delta/2}\rangle<+\infty.
\end{equation}
\end{lemma}
\proof
Suppose that $L>1$ is arbitrary and $\psi_L(x)$ is given by \eqref{012712}. An analogous formula defines also $\chi_L(x)$.
Let $M_t^{(L)}$ be given by the analogue of  \eqref{mart}, where $\psi$ and $\chi$ are replaced by $\psi_L$ and $\chi_L$ respectively. Thanks to  \eqref{052712} it is easy to verify that the function 
$$
H_K^{(L)}(x)\coloneqq  \mathbb{E}_x\left[\frac{1}{K}[M^{(L)}_{K}]^2-\sigma^{2}\right]
$$
is Lipschitz on $B_R(x_0)$ for  any $R>0$ and $x_0\in E$ and, due to hypothesis H2)
$$
\lim_{L\to+\infty}\|H_K^{(L)}-H_K\|_{\infty,B_R(x_0)}=0.
$$
This proves that $H_K\in C(E)$.

Considerations similar  to those made in the proof of estimate  \eqref{010501} lead to
\begin{eqnarray}
&&
\langle\mu_{0}Q_{\ell}^K,|H_K|^{1+\delta/2}\rangle\nonumber\\
&&
\leq \frac{C}{\ell}\left\{(\|\chi\|_{L}+1)^{2+\delta}\sum_{m=1}^{\ell}\left[\bbE\rho^{2+\delta}_{x_0}(X_{m K})+
\bbE\rho^{2+\delta}_{x_0}(X_{(m-1) K})\right]\right.\nonumber\\
&&
\left.+[\|\psi\|_{L}+|\psi(x_0)|+1]^{2}\sum_{m=1}^{\ell}\int_{(m-1)K}^{mK}[\bbE\rho^{2+\delta}_{x_0}(X_{s})+1]ds\right\}
\end{eqnarray}
and the expression on the right hand side remains bounded, as $\ell\to+\infty$, thanks to assumption H3). Thus \eqref{square-est} follows.
\qed

Using the above lemma we conclude that for any $K$
\begin{eqnarray*}
\lim_{\ell\to+\infty}\langle \mu_{0}Q_{\ell}^K,|H_K|\rangle= \langle \mu_*, |H_K|\rangle.
\end{eqnarray*}
Since $\mu_{*}$ is ergodic under the Markovian dynamics, from Birkhoff's ergodic theorem we obtain that the limit of the expression on the right hand side, as $K\to+\infty$, equals $0$,
provided that
$
\si^2:=\mathbb{E}_{\mu_*}M_1^2.
$
This ends the inspection of hypothesis M2).

\subsubsection*{Condition M3)} We can rewrite the expression appearing under the limit in \eqref{m3} as being equal to
\begin{eqnarray*}
\frac 1K\sum_{j=0}^{K-1}\langle \mu_{0}Q_{\ell}^K, G_{\ell,j}\rangle
\end{eqnarray*}
where
$$
G_{\ell,j}(x)\coloneqq \mathbb{E}_x\left[1+Z_{j+1}^2,|M_{j}|\geq\epsilon\sqrt{\ell K}\right].
$$
It suffices only to prove that 
\begin{eqnarray}
\label{G-j}
\limsup_{\ell\rightarrow\infty}\langle \mu_{0}Q_{\ell}^K, G_{\ell,j}\rangle=0\quad\forall\,j=0,\ldots,K-1.
\end{eqnarray}
From the Markov inequality we obtain
\begin{eqnarray*}
&&\mathbb{P}_x\Big[|M_{j}|\geq \epsilon\sqrt{\ell K}]\leq \frac{\mathbb{E}_{x}|M_{j}|}{\epsilon\sqrt{\ell K}}\\&&\leq
\frac{1}{\epsilon\sqrt{\ell K}}\left\{\mathbb{E}_{x}|\chi(X_{j})-\chi(x)|+\left|\chi_j(x)\right|\right\}.\nonumber
\end{eqnarray*}
Using Lemmas \ref{lab21},  \ref{lab38}   and H2) we obtain that for any $x_0\in E$
\begin{eqnarray}
\label{012201-2011}
\sup_{x\in B_R(x_0)}\mathbb{P}_{x}\Big[|M_{j}|\geq \epsilon\sqrt{\ell K}|\Big]\leq \frac{C}{\sqrt{\ell K}}.
\end{eqnarray}
Estimating as in \eqref{012611} we get
\begin{eqnarray}
\label{022611}
&&\sup_{x\in B_R(x_0)}\mathbb{E}_{x}\left[Z^{2}_{j+1},| M_{j}|\geq\epsilon\sqrt{\ell K}\right]\\
&&
\leq\,
2\left\{\sup_{x\in B_R(x_0)}\mathbb{E}_{x}\left\{\left[\chi(X_{j+1})-\chi(X_{j})\right]^2,| M_{j}|\geq\epsilon\sqrt{\ell K}\right\}\right.\nonumber\\
&&+
\left.\sup_{x\in B_R(x_0)}\mathbb{E}_{x}\left\{\Big[\int^{j+1}_{j}\psi (X_{s})ds\Big]^{2},| M_{j}|\geq\epsilon\sqrt{\ell K}\right\}\right\}\nonumber
\\
&&
\le C\sup_{t\in[0,K]}\sup_{x\in B_R(x_0)}\mathbb{E}_{x}\left[\rho_x^{2}(X_t),| M_{j}|\geq\epsilon\sqrt{\ell K}\right].\nonumber
\end{eqnarray}
for some constant $C$ independent of $\ell$. The utmost right hand side of \eqref{022611} can be further estimated by 
\begin{eqnarray*}
&&C\sup_{t\in[0,K]}\sup_{x\in B_R(x_0)}\left\{\mathbb{E}_{x}\left[\rho_x^{2+\delta}(X_t),| M_{j}|\geq\epsilon\sqrt{\ell K}\right]\right\}^{2/(2+\delta)}\\
&&\times\left\{\sup_{x\in B_R(x_0)}\mathbb{P}_{x}\left[| M_{j}|\geq\epsilon\sqrt{\ell K}\right]\right\}^{\delta/(2+\delta)}
\end{eqnarray*}
Using \eqref{012201-2011} and hypothesis H2) we conclude that
\begin{equation}
\label{022201-2011}
\lim_{\ell\to+\infty}\sup_{x\in B_R(x_0)}G_{\ell
,j}(x)=0,\quad\forall\,x_0\in E, R>0.
\end{equation} 
To obtain \eqref{G-j} it suffices to prove only  that for  $\delta>0$ as in H3) we have
 \begin{equation}
\label{032201-2011}
 \limsup_{\ell\rightarrow\infty}\langle \mu_{0}Q_{\ell}^K, G_{\ell,j}^{1+\delta/2} \rangle<\infty,\quad\forall\,K\ge1, 
\,0\le j\le K-1.
\end{equation} 
Note that 
\begin{eqnarray}
\langle \mu_{0}Q_{\ell}^K, G_{\ell,j}^{1+\delta/2} \rangle\le\mathbb{E}_{\mu_{0}Q_{\ell}^K}(1+Z_{j+1}^2 )^{1+\delta/2}\label{lab48}
\end{eqnarray}
Using Lemmas  \ref{lab21},  \ref{lab38}   and hypothesis H3) we can estimate the expression on the right hand side by
 \begin{eqnarray}
\sup_{t\ge0}\bbE_{ \mu_{0}Q_{\ell}^K} \rho^{2+\delta}_{x_0}(X_t)\le
A_*
\label{lab48a}
\end{eqnarray}
for some $x_0$ and  $A_*$ as in the statement of H3).
Thus \eqref{032201-2011} follows.

\section{Applications.}
\label{dissipative} 
\subsection{Stochastic differential equation with a dissipative drift}

\label{dissipative1} 
In this section we consider an example of a stochastic differential equation with a dissipative drift coming from Section 6.3.1, p. 108  of \cite{dpz2}.
Suppose that  $(H,|\cdot|)$ is a separable Hilbert space, with the scalar product $\langle\cdot,\cdot\rangle$ and $(-A):D(A)\to H$ is the generator of $\{S_t,\,t\ge0\}$ -  a strongly continuous, analytic semigroup  of operators on $H$, for which
there exists $\omega_1\in\bbR$ such that  $\{e^{\omega_1 t}S_t,\,t\ge0\}$
is a semigroup of contractions. 
The above implies, in particular, that
\begin{eqnarray}
\langle Ax,x\rangle\geq \omega_1|x|^2,\quad x\in D(A).
 \label{lab74}
\end{eqnarray}
Hence any $\lambda>-\omega_1$ belongs to the resolvent set of $A$ and we can define a bounded operator $(\lambda+A)^{-1}$.

Next,  we suppose that $F:H\to H$ is Lipschitz, i.e. there is $L_F>0$ such that
 \begin{eqnarray}
| F(y+z)-F(z)|\leq L_F|y|\label{lip}
\end{eqnarray}
and for some $\omega_2\in\bbR$ such that
\begin{eqnarray}
 \omega:=\omega_1+\omega_2 >0\label{lab70}
\end{eqnarray}
we have
\begin{eqnarray}
\langle F(y+z)-F(z),y\rangle\leq -\omega_2|y|^2\label{lab69},\quad\forall\,y,z\in H.
\end{eqnarray}

 Suppose that $\{e_i,\,i\ge1\}$ is an orthonormal base in $H$ and
 $\{B_p(t),\,t\ge0\}_{p\ge1}$ is a collection of independent, standard, one-dimensional Brownian motions over $(\Omega ,\mathfrak{F},\mathbb{P})$ that are non-anticipative with respect to a filtration $\{{\frak F}_{t},t\ge0\}$ of sub $\si$-algebras of $\frak F$. Let  $\{\gamma_p,\,p\ge1\}$ be a sequence of reals such that
$
 \sum_{p=1}^\infty \gamma_p^2<\infty,
$ then 
\begin{equation*}
W(t):=\sum_{p=1}^{+\infty}\gamma_p B_p(t)e_p,\quad t\ge0
\end{equation*}
 is an $H$-valued Wiener process with the covariance operator 
\begin{equation}
\label{Q}
Qx=\sum_{p=1}^{+\infty}\gamma_p^2\langle x,e_p\rangle e_p,\quad x\in H.
\end{equation}
 Let
$$
Z_t:=\int_0^t S_{t-s} d W(s)
$$
 be the stochastic convolution process defined in Section 5.1.2 of \cite{dpz1}. It is  Gaussian and $H$-continuous.   We assume that
  \begin{eqnarray}
\sup_{t\geq 0}\int_0^t\mbox{Trace}(S_s^*QS_s)ds< \infty \label{lab73a},
\end{eqnarray}
 which in turn guarantees that 
 \begin{eqnarray}
\sup_{t\geq 0}\mathbb{E}|Z_t|^2< \infty \label{lab73}.
\end{eqnarray}

We consider the following It\^o stochastic differential equation
\begin{eqnarray}
\left\{
\begin{array}{l}
dX_t(\xi)=[-AX_t(\xi)+F(X_t(\xi))]dt+dW(t)\\
X_0(\xi)=\xi,
\end{array}
\right. \label{lab50}
\end{eqnarray}
where $\xi$ is an $\mathfrak{F}_0$-measurable, $H$-valued, random element. When $\xi$ is obvious from the context we shall abbreviate and write $X_t$, instead of $X_t(\xi).$ We shall also write $X_t(x)$ when  $\xi=x$ with probability one.

 A solution of \eqref{lab50} is understood in the mild sense, see  p. 81 of \cite{dpz2}, i.e.  $\{X_t,\,t\ge0\}$ is an $\{{\frak F}_t,\,t\ge0\}$ adapted, continuous trajectory process, such that 
 \begin{eqnarray*}
X_t=S_t\xi +\int_0^t S_{t-s}F(X_s)ds+Z_t,\quad t\ge0,
\end{eqnarray*}
$\bbP$ a.s. 
We shall assume that there exists $\delta>0$ such that
\begin{equation}
\label{two+}
\bbE|\xi|^{2+\delta}<+\infty.
\end{equation}

  It is known, see Theorem 5.5.11 of \cite{dpz2}, that under the hypotheses made about $A$, $F$ and $W(t)$, for each $x\in H$ there exists a unique mild solution $X_t(x)$ of \eqref{lab50}. The solutions  $\{X_t(x),\,t\ge0\}$, $x\in H$ form a Markov family that corresponds to a Feller transition semigroup. Moreover, there exists a unique invariant probability measure $\mu_*$ for the above Markov family such that for any random element $\xi$ the laws of $X_t(\xi)$ converge to $\mu_*$, in the sense of the weak convergence of measures, see Theorem 6.3.3, p. 109 of \cite{dpz2}.

Our main theorem in this section is the following.
\begin{thm}
\label{thm012901}
Suppose that $\psi\in\,${\em Lip}$(H)$ and $\{X_t(\xi),\,t\ge0\}$ is the solution of \eqref{lab50}.
Then, under the  assumptions made above, the functional $\int_0^t\psi(X_s(\xi))ds$,  satisfies the conclusions $1)-3)$ of Theorem $\ref{lab3}$.
\end{thm}
\begin{proof} Our calculations are based on a similar computation made in \cite{peszat} in the context of an equation with a L\'evy noise.
Define the {\em  Yosida approximation} of $A_{\omega_1}:=A-\omega_1$ as a bounded operator 
$$
A_{\alpha,\omega_1} \coloneqq -\alpha^{-1}[(I+\alpha A_{\omega_1})^{-1}-I]=A_{\omega_1}(I+\alpha A_{\omega_1})^{-1}
$$
The associated semigroup $\{S_{t,\alpha} , t\geq 0\}$ strongly converges to $\{e^{\om_1 t}S_t, t\geq 0\}$, as  $\alpha \rightarrow 0\!+$, see Theorem 3.5 of \cite{nagel}. Let $\hat A_{\alpha,\omega_1}:=A_{\alpha,\omega_1}+\omega_1I$.
\begin{lemma}\label{lab63}
Suppose that $A$ satisfies \eqref{lab74}.
 Then, for any  $\alpha > 0$  
$$
\langle  \hat A_{\alpha,\omega_1}x,x\rangle\geq \omega_1 |x|^2,\quad\forall\,x\in H.
$$
\end{lemma}
\begin{proof}
It suffices only to show that for any $x\in D(A)$ and $y=[1+\alpha (A-\omega_1)]x$ 
we have $\langle A_{\alpha,\omega_1} y,y\rangle\geq 0.$
 Indeed
 \begin{eqnarray*}
 &&\langle A_{\alpha,\omega_1} y,y\rangle=\langle(A-\omega_1)[1+\alpha(A-\omega_1)]^{-1}y,y\rangle\\
 &&
 =\langle(A-\omega_1) x,[1+\alpha(A-\omega_1)x]\rangle\\
 &&=\langle(A-\omega_1) x,x\rangle+\alpha|(A-\omega_1)x|^2\geq 0.
 \end{eqnarray*}
\end{proof}
From the above lemma, \eqref{lab74} and \eqref{lab69} we conclude that.
\begin{cor} We have
\begin{eqnarray}
\langle (-\hat A_{\alpha,\omega_1})y+F(y+z)-F(z),y\rangle\leq -\omega|y|^2\label{lab69a},
\end{eqnarray}
for all $y\in D(A)$, $z\in H$,
\end{cor}

Denote by $X_{t,\alpha}$ the solution of 
\begin{eqnarray}
\left\{
\begin{array}{l}
dX_{t,\alpha}=\left[-\hat A_{\alpha,\omega_1} X_{t,\alpha}+F(X_{t,\alpha})\right]dt+dW(t)\\
X_{0,\alpha}=\xi,
\end{array}
\right. \label{lab51}
\end{eqnarray}
Since the drift on the right hand side is Lipschitz and the noise is additive, this equation  \eqref{lab51} has a unique strong solution, i.e. the $\{{\frak F}_t,\,t\ge0\}$ adapted, $H$-continuous trajectory process $X_{t,\alpha}$ such that 
$$
X_{t,\alpha}=\xi+\int_0^t\left[-\hat A_{\alpha,\omega_1} X_{s,\alpha}+F(X_{s,\alpha})\right]ds+W(t)
$$
$\bbP$ a.s. One can show, see \cite{dpz2}, p. 81, that  $\lim_{\alpha \rightarrow 0\!+}\sup_{t\in[0,T]}|X_{t,\alpha}-X_t|=0$, $\bbP$ a.s. for any $T>0$.
Consider the linear equation with an additive noise
\begin{eqnarray}
\left\{
\begin{array}{l}
dZ_t(\xi)=-AZ_t(\xi)dt+dW(t)\\
Z_0(\xi)=\xi .\label{lab59}
\end{array}
\right.
\end{eqnarray}
It has a unique mild solution, given by formula,
\begin{eqnarray*}
Z_t(\xi)=S_t\xi+\int_0^t S_{t-s}\,d W(s) \quad t\geq 0.
\end{eqnarray*}
Denote by $Z_{t,\alpha}(0)$ the strong solution of  
\begin{eqnarray}
\left\{
\begin{array}{l}
dZ_{t,\alpha}(0)=-A_{\alpha,\omega_1} Z_{t,\alpha}(0)dt+dW(t)\\
Z_{0,\alpha}=0 .\label{lab59a}
\end{array}
\right.
\end{eqnarray}
To abbreviate the notation we shall write $Z_{t,\alpha}$, $Z_t$ instead of $Z_{t,\alpha}(0)$ and $Z_t(0)$, respectively.
We have $\lim_{\alpha \rightarrow 0\!+}\sup_{t\in[0,T]}|Z_{t,\alpha}-Z_t|=0$, $\bbP$ a.s. for any $T>0$.
Define
\begin{eqnarray*}
Y_{t,\alpha}=X_{t,\alpha}-Z_{t,\alpha}.
\end{eqnarray*}
Then,
\begin{eqnarray}
\frac{dY_{t,\alpha}}{dt}=-\hat A_{\alpha,\omega_1} Y_{t,\alpha}+F(Y_{t,\alpha}+Z_{t,\alpha})\label{lab52}.
\end{eqnarray}
For any $\epsilon>0$ define $|Y_{t,\alpha}|_{\epsilon}:=\sqrt{|Y_{t,\alpha}|^2+\epsilon^2}$. Then
\begin{eqnarray}
\frac{d}{dt}|Y_{t,\alpha}|_{\epsilon}=\Big\langle\frac{dY_{t,\alpha}}{dt},\frac{Y_{t,\alpha}}{|Y_{t,\alpha}|_{\epsilon}}\Big\rangle\label{lab53}
\end{eqnarray}
Substituting from  \eqref{lab52} into the right hand side of \eqref{lab53} we get
\begin{eqnarray}
&&\frac{d}{dt}|Y_{t,\alpha}|_{\epsilon}=-\frac{1}{|Y_{t,\alpha}|_{\epsilon}}\langle \hat A_{\alpha,\omega_1} Y_{t,\alpha},Y_{t,\alpha}\rangle+\frac{1}{|Y_{t,\alpha}|_{\epsilon}}\langle F(Y_{t,\alpha}+Z_{t,\alpha}),Y_{t,\alpha}\rangle\nonumber\\
&&=\frac{1}{|Y_{t,\alpha}|_{\epsilon}}[\langle (-\hat A_{\alpha,\omega_1} )Y_{t,\alpha}+ F(Y_{t,\alpha}+Z_{t,\alpha})\!\!-\!\!F(Z_{t,\alpha}),Y_{t,\alpha}\rangle]\nonumber\\
&&+\frac{1}{|Y_{t,\alpha}|_{\epsilon}}\langle F(Z_{t,\alpha}),Y_{t,\alpha}\rangle\label{lab54}
\end{eqnarray}
Using \eqref{lab69a} and the Cauchy-Schwarz inequality  we conclude that 
\begin{eqnarray*}
\frac{d}{dt}|Y_{t,\alpha}|_{\epsilon}
\leq -\omega| Y_{t,\alpha}|+|F(Z_{t,\alpha})|.
\end{eqnarray*}
Letting $\epsilon\to0+$ we get
\begin{eqnarray*}
|Y_{t,\alpha}|-|Y_{0,\alpha}|\le\int_0^t\left\{ -\omega| Y_{s,\alpha}|+|F(Z_{s,\alpha})|\right\}ds.
\end{eqnarray*}
Removing the Yosida regularization, by sending $\alpha\to 0+$, we get
\begin{eqnarray*}
|Y_t|-|Y_0|\le-\omega\int_0^t | Y_s|ds+\int_0^t |F(Z_s)|ds,\quad t\ge0.
\end{eqnarray*}
From this we conclude, via Gronwall's inequality, that
\begin{eqnarray}
|Y_t|\leq e^{-\omega t}|\xi|+\int_0^t e^{-\omega (t-s)}|F(Z_s)|ds\label{lab55}.
\end{eqnarray}

Consider now  $X_t(\xi)$ and $X_t(\overline{\xi})$ the two solutions of \eqref{lab50} corresponding to the initial conditions $\xi$ and  $\overline{\xi}$. We conclude that their difference $\Delta_t:=X_t(\xi)-X_t(\overline{\xi})$ satisfies equation
\begin{eqnarray}
\left\{
\begin{array}{l}
\dfrac{d\Delta_t}{dt}=-A\Delta_t+F(\Delta_t+X_t(\bar \xi))-F(X_t(\bar \xi))\\
\\
\Delta_0=\xi-\bar \xi.
\end{array}
\right. \label{lab50b}
\end{eqnarray}
An analogous calculation to the one carried out above, using the Yosida approximation and the dissipativity condition,  yields
\begin{eqnarray}
|\Delta_t|-|\Delta_0|\le-\omega\int_0^t | \Delta_s|ds\label{lab56a}.
\end{eqnarray}
Thus,
\begin{eqnarray}
|\Delta_t|\le e^{-\om t}|\Delta_0|,\quad \forall\,t\ge0.
\label{lab56b}
\end{eqnarray}

The proof of Theorem \ref{thm012901} consists in the inspection of the hypotheses of our main Theorem \ref{lab3}. From properties of a mild solution
of \eqref{lab50} we conclude that the semigroup corresponding to the Markov family $X_t(x)$ is Feller and stochastically continuous, so H0) holds.

\paragraph{\textbf{Verification of H2)}}

From  \eqref{lab55} and the Lipschitz property of $F$ we conclude that there exists $C>0$ such that
\begin{eqnarray}
\label{013001}
|Y_t(x)|\leq e^{-\omega t}|x|+C\int_0^t e^{-\omega (t-s)}(1+|Z_s|)ds
\end{eqnarray}
hence,  there is a constant $C>0$ such that
\begin{eqnarray}
\label{012901}
\mathbb{E} |Y_t(x)|^{2+\delta}\leq C\left[ |x|^{2+\delta} +\int_0^t e^{-\omega (2+\delta)(t-s)}(1+\mathbb{E}|Z_s|^{2+\delta})ds\right]
\end{eqnarray}
for all $t\ge0$. Thus, $\sup_{|x|\leq R}\mathbb{E} |Y_t(x)|^{2+\delta} <\infty$ 
and since
\begin{equation}
\label{022901}
X_t(x)=Y_t(x)+Z_t
\end{equation}
we conclude that
\begin{equation}
\label{032901}
\sup_{|x|\leq R}\mathbb{E} |X_t(x)|^{2+\delta}<\infty
\end{equation}
 for any $R>0$. This implies H2).

\paragraph{\textbf{Verification of H3)}}
Suppose that   $\mu_{0}$ is the law of $\xi$. 
Then,
$$
\sup_{t\geq 0}\mathbb{E}|X_t(\xi)|^{2+\delta}\leq C\,\left\{\sup_{t\geq 0}\mathbb{E}|Y_t(\xi)|^{2+\delta}+\sup_{t\geq 0}\mathbb{E}|Z_t|^{2+\delta}\right\}<\infty .
$$
From \eqref{lab73} and the fact that $\{Z_t,\,t\ge0\}$ is Gaussian it follows that $\sup_{t\geq 0}\mathbb{E}|Z_t|^{2+\delta}< \infty .$
Using \eqref{013001} we conclude easily  that there exists $C>0$ such that
\begin{eqnarray*}
&&\mathbb{E}|Y_t(\xi)|^{2+\delta}\\
&&
\leq C\left\{\,e^{-(2+\delta)\omega t}\mathbb{E}|\xi|^{2+\delta}+\mathbb{E}\left[\int_0^t e^{-\omega (t-s)}(1+|Z_s|)ds\right]^{2+\delta}\right\}.
\end{eqnarray*}
Thus, from the above and  \eqref{022901} we get
$$
\sup_{t\ge0}\mathbb{E}|X_t(\xi)|^{2+\delta}<+\infty
$$ and therefore H3) holds.

\paragraph{\textbf{Verification of  H1)}}

Estimate \eqref{013001} together with formula \eqref{022901} guarantee that the space ${\cal P}_1$ is preserved under $P^t$.
Suppose that $X_t(\xi)$ and  $X_t(\overline{\xi})$ are two processes that at  $t=0$ equal $\xi$ and $\bar \xi$, with the laws   $\mu_1$ and  $\mu_2$, respectively. From \eqref{lab56b} we get that for any $\psi \in$Lip$(E)$ 
\begin{eqnarray}
&&|\mathbb{E}\psi (X_t(\xi))-\mathbb{E}\psi (X_t(\overline{\xi}))|\leq \|\psi\|_L\mathbb{E} |X_t(\xi)- X_t(\overline{\xi})|\nonumber\\
&&\leq\|\psi\|_L e^{-\omega t} \mathbb{E}|\xi-\overline{\xi}|\label{lab57}
\end{eqnarray}
Taking the supremum  over all $\psi$ such that $\|\psi\|_L\le 1$  and the infimum over all couplings $(\xi,\bar\xi)$ whose marginals equal $\mu_1$, $\mu_2$, correspondingly on the left and  right hand sides, we obtain
\begin{eqnarray*}
&&d_1 (\mu_1 P^t,\mu_2 P^t)\le e^{-\omega t}d_1(\mu_1,\mu_2),\quad \forall\,t\ge0.
\end{eqnarray*}
Thus, H1) holds.
\end{proof}

\subsection{Two dimensional Navier-Stokes system of equations with Gaussian forcing}

\label{sec7.2}
Let $\mathbb{T}^2$ be a two dimensional torus  understood here  as the product of two copies of $[-1/2,1/2]$ with identified endpoints.
Suppose that $u(t,x)=(u^1(t,x),u^2(t,x))$  and $p(t,x)$ are respectively a  two dimensional vector valued and a scalar valued field, defined for $(t,x)\in[0,+\infty)\times \bbT^2$. They  satisfy  the  two dimensional Navier--Stokes equation system with  forcing $F(t,x)=(F^1(t,x),F^2(t,x))$, i.e.
\begin{equation}
\label{E11}
\begin{aligned}
\partial_t u^{i}(t,x)&+ u(t,x)\cdot \nabla_x u^{i}(t,x)&&\\
&=  \Delta_x u^{i}(t,x) -\partial_{x_i} p(t,x)+ F^i(t,x),\quad i=1,2\\
&\sum_{j=1}^2\partial_{x_j}u^j(t,x)= 0,\\
& u(0,x)=u_0(x).
\end{aligned}
\end{equation}
Here $\Delta_x$, $\nabla_x$ denote the Laplacian and  gradient operators and $u_0(x)$ is the initial data. We shall be concerned with the asymptotic description of functionals of the form
$\int_0^t\psi(u(s))ds$ in case $F(t,x)$ is a  Gaussian white noise  in time and $\psi$ is a Lipschitz continuous observable on an appropriate state space.
Below, we recap briefly some  of the results of \cite{HM1}. Assume that $(\Om,{\frak F},\{{\frak F}_t,\,t\ge0\},\bbP)$ and $\{W(t),\,t\ge0\}$ are a filtered probability space and a Wiener process on Hilbert space $H=L^2_0(\bbT^2)$ - made of  square integrable, zero mean functions -  as in the previous section equipped with the norm $|\cdot|$.  The orthonormal base  $e_p$ appearing in \eqref{Q} is given by 
$$
e_p(x):=\exp\{2\pi ip\cdot x\},\quad p=(p^1,p^2)\in\bbZ^2_*:=\bbZ^2\setminus\{(0,0)\}
$$ (we abuse slightly the notation admitting a two parameter index). 

We rewrite the system \eqref{E11} using the vorticity formulation, i.e. we write an equation for the scalar, called vorticity, 
$$
\om(t):={\rm rot}\,u(t)=\partial_{x_2}u^1(t)-\partial_{x_1}u^2(t).
$$
 It satisfies then an It\^o stochastic differential equation
\begin{eqnarray}\label{E25}
&&d \om(t;w) =[\Delta_x\om(t;w) +B(\om(t;w))]d t + d  W(t), \\
&&
 \om(0;w)=w\in H.\nonumber
\end{eqnarray}
Here 
$$
B(\om):=-\sum_{j=1}^2{\cal K}^j(\om)\partial_{x_j} \om,
$$
with ${\cal K}:=({\cal K}^1,{\cal K}^2)$ given by
$
{\cal K}(\om)=\sum_{p\in\bbZ^2_*}p^{\perp}|p|^{-2}\langle \om,e_p\rangle e_p
$ and $p^{\perp}=(p^2,-p_1)$. The existence and  uniqueness result and continuous dependence of solutions  on the initial data for \eqref{E25} can be found in  e.g. \cite{dpz1}. As a result the solutions $\{\om(t;w),\,t\ge 0\}$ determine a Feller, Markov family of $H$-valued processes. Denote by $\{P^t,\,t\ge0\}$ the corresponding transition probability semigroup and its dual acting on measures. 
Following \cite{HM1} we adopt the non-degeneracy of the noise assumption that can be stated as follows:
\begin{itemize}
\item[ND)] the set ${\cal Z}:=[p:\gamma_p\not=0]$  is finite, symmetric with respect to $0$, generates $\bbZ^2$, i.e.  integer linear combinations of elements of ${\cal Z}$ yield the entire $\bbZ^2$ and there exists at least two $p_1,p_2$ with $|p_1|\not=|p_2|$ such that $\gamma_{p_i}\in{\cal Z}$ for $i=1,2$.
\end{itemize}
For any $\eta>0$ define also $V:H\to[0,+\infty)$ by $V(w):=\exp\left\{\eta|w|^2\right\}$ and a metric
$$
\rho(w_1,w_2)=\inf_{\gamma}\int_0^1V(\gamma(s))|\dot\gamma(s)|ds,
$$
where infimum is taken over all $C^1$ smooth functions $\gamma(s)$ such that $\gamma(0)=w_1$, $\gamma(1)=w_2$. It is clear that $\rho$ metrizes the strong topology of $H$ and is equivalent with the metric induced by the norm on any  finite ball. 
Denote by $C^1_\eta(H)$ the space of functionals $\psi:H\to\bbR$ that possess Frechet derivative $D\psi$ satisfying
$$
\|\psi\|_{\eta}:=\sup_{u\in H}e^{-\eta|w|^2}\left(|\psi(u)|+\|D\psi(u)\|\right)<+\infty.
$$
It is elementary to verify the following.
\begin{prop}
\label{prop012511}
We have
$$
\psi(w_2,w_1)\le \|\psi\|_{\eta}\rho(w_2,w_1),\quad\forall\,w_1,w_2\in H.
$$ 
\end{prop}
Denote by $d_1(\cdot,\cdot)$ the corresponding Wasserstein metric on ${\cal P}_1(H,\rho)$ - the space of probability measures on $H$ having the first moment with respect to metric $\rho$. The following theorem summarizes  the  results   of  \cite{HM,HM1} that are of particular interest for us 
\begin{thm}\label{T1}
Under the assumptions made above the following hold:
\begin{itemize} 
\item[1)] there exists $\nu_0$ such that for any $\nu\in(0,\nu_0]$ and $T>0$ there exist $C>0$ for which
\begin{equation}
\label{012011}
\bbE\exp\left\{\nu\om^2(t)\right\}\le C\bbE\exp\left\{\nu e^{-t}\om^2(0)\right\},\quad\forall\,t\ge0,
\end{equation}
\item[2)] we have $P^t({\cal P}_1(H,\rho))\subset {\cal P}_1(H,\rho)$ for all $t\ge0$ and there exist $\hat c,\gamma>0$ such that estimate \eqref{022511} holds.
\end{itemize}
\end{thm}
Part 1) of the theorem follows from estimate (A.5) of \cite{HM}, while part 2) is a consequence of Theorem 3.4 of \cite{HM1}. Choosing $\eta>0$, in the definition of metric  $\rho(\cdot,\cdot)$,  sufficiently small we conclude from part 1) of Theorem \ref{T1} that hypotheses H2) and H3) hold. Part 2) allows us to conclude hypothesis H1). As we have already mentioned hypothesis H0) concerning Feller property  also holds, therefore by virtue of Theorem \ref{lab3} we conclude the following.
\begin{thm}\label{T1a}
Suppose that $\psi\in C^1_\eta(H)$. Then,  the functional $\int_0^t\psi(\om(s))ds$ satisfies the conclusion 1) - 3) of Theorem $\ref{lab3}$.
\end{thm}

We should also mention that  the proof of the central limit theorem, in the perhaps most interesting, from the physical viewpoint, case of an additive functional of the point evaluation of the Eulerian velocity $u(s,x):={\cal K}(\om(t))(x)$ is slightly more involved.  The respective observable  is not  Lipschitz and the results of the present paper are not directly applicable. In that case one can use the regularization result of \cite{HM1}, see Proposition 5.12.

%

\appendix
\section{Proof of the central limit theorem for martingales}

\label{appA}

{\em Proof of \eqref{M2b}.}
 Suppose first that $N=\ell K$ for some positive integers $K, \ell$. Then,
 \begin{eqnarray*}
 &&\Big|\mathbb{E}\left[\frac{1}{N}<\!M\!>_{N}\right]-\sigma^2\Big|
 \\
 &&
 \le\frac{1}{\ell}\sum_{m=1}^{\ell}\mathbb{E}\Big|\frac{1}{K}\bbE\left[<\!M\!>_{mK-1}-<\!M\!>_{(m-1)K}\Big|{\frak F}_{(m-1)K} \right]-\sigma^{2}\big |\to 0
  \end{eqnarray*}
as $\ell\to+\infty$ and then $K\to+\infty$ (in this order).
When $N=\ell K+r$ for some $0\le r\le K-1$ we can use the above result and \eqref{sublinear} to conclude \eqref{M2b}.

{\em Proof of \eqref{052501}.}
The following argument is a modification of the proof coming from Chapter 2 of  \cite{klo}.   Choose an arbitrary $\rho>0$.
Recall that for all $a\in\bbR\setminus\{0\}$ we can write
\begin{equation}
e^{ia} \;=\; 1 \;+\; ia  \;-\; a^2/2  \;-\; R(a) a^2 \label{lab312}
\end{equation}
where $R(0)=0$ and
$$
R(a)\;:=\; a^{-2} \int_0^a da_1 \int_0^{a_1} (e^{ix}-1) \, dx\; \mbox{ for }a\not=0.
$$
It satisfies $|R(a)|\le 1$ and
\begin{equation}
\label{012501}
\lim_{a\to0}R(a)=0.
\end{equation}
To simplify the notation we introduce the following abbreviations
\begin{eqnarray}
\label{072501}
&&A_j:=(\theta/\sqrt{N})  Z_{j+1}, \quad R_j:=R (
A_j ),\\
&&
\Delta_j:=\mathbb{E}e^{i(\theta/\sqrt{N}) M_{j+1}}  -
\mathbb{E}e^{i(\theta/\sqrt{N}) M_{j}} ,\\
&&
e_{j,N}:=e^{i(\theta/\sqrt{N}) M_{j}} . 
\end{eqnarray}
Using the fact that
$\mathbb{E} [ Z_{j+1} \,\vert\, \mathfrak{F}_j]$ $= 0$ we can write
\begin{eqnarray*}
&&\mathbb{E}e_{j+1,N}  \;=\; \mathbb{E}\Big [e_{j,N}  \Big\{ 1 \;+\;
\mathbb{E}\Big [e^{iA_j} \;-\; 1
\;-\; A_j \,
\Big\vert\, \mathfrak{F}_j\Big]\, \Big\}\, \Big].
\end{eqnarray*}
From \eqref{lab312} we get
\begin{eqnarray}\label{lab31}
\Delta_j = -\;  \frac {\theta^2}{2N} \, \mathbb{E}\Big [
e_{j,N} \, Z_{j+1}^2 \Big] 
\;- \;  \frac {\theta^2}{N} \, \mathbb{E}\Big[e_{j,N} \, Z_{j+1}^2 R_j\, \Big] \; .
\end{eqnarray}
Hence,
\begin{eqnarray}\label{lab32}
&& e^{\theta^2\sigma^2(j+1)/(2N)} \, \mathbb{E} e_{j+1,N}  \;-\; e^{\theta^2\sigma^2j/(2N)} \,
 \mathbb{E}e_{j,N}  \\
&& =\; e^{\theta^2\sigma^2(j+1)/(2N)} \Delta_j +\; e^{\theta^2\sigma^2(j+1)/(2N)}
\Big[ 1- e^{-\theta^2\sigma^2/(2N)} \Big]
\mathbb{E} e_{j,N} \;\nonumber.
\end{eqnarray}
Using \eqref{lab31} we conclude that  the right hand side of the above
equation equals 
\begin{eqnarray*}
&& e^{\theta^2\sigma^2(j+1)/(2N)} \Big\{-\;  \frac {\theta^2}{2N} \, \mathbb{E}\Big [
e_{j,N} \, Z_{j+1}^2 \Big] -\frac {\theta^2}{N}\,\mathbb{E}\Big[e_{j,N}\,Z_{j+1}^2 \, R_j\Big] \Big\}\\
&&+\; e^{\theta^2\sigma^2(j+1)/(2N)}
\Big[ 1- e^{-\theta^2\sigma^2/(2N)} \Big]
\mathbb{E} e_{j,N}\\
&&= -\frac{\theta^2}{2N} \, e^{\theta^2\sigma^2(j+1)/(2N)}
\, \mathbb{E}\Big [ e_{j,N} \,(Z_{j+1}^2-\sigma^2) \Big] \\
&&  -\; \frac{\theta^2}{N} \, e^{\theta^2\sigma^2(j+1)/(2N)}
\mathbb{E}\Big [e_{j,N} \, Z_{j+1}^2\, R_j \Big] \\
&& +\; e^{\theta^2\sigma^2(j+1)/(2N)}
\Big\{ 1-\frac{(\theta\sigma)^2}{2N} -e^{-\theta^2\sigma^2/(2N)}\Big\} \,
\mathbb{E}\Big [e_{j,N} \Big]\;.
\end{eqnarray*}
 Summing up over $j$ from $0$ to $N-1$ and  ($M_0=0$)
we get
\begin{eqnarray}\label{lab33}
&& e^{(\theta^2\sigma^2)/2}\,\mathbb{E}\,e^{i(\theta\sqrt{N})M_{N}}-1\nonumber\\
&&=-\; \frac {\theta^2}{2N} \sum_{j=0}^{N-1}
e^{\theta^2\sigma^2(j+1)/(2N)} \, \mathbb{E}\Big [
e_{j,N} (Z_{j+1}^2-\sigma^2) \Big]\nonumber \\
&&   -\;  \frac {\theta^2}{N} \sum_{j=0}^{N-1}
e^{\theta^2\sigma^2(j+1)/(2N)}
\mathbb{E}\Big[ e_{j,N} \, Z_{j+1}^2 \,
R_j \Big]  \\
&&  +\; \sum_{j=0}^{N-1} e^{\theta^2\sigma^2(j+1)/(2N)}
\Big\{ 1-\frac{(\theta\sigma)^2}{2N} -
e^{-\theta^2\sigma^2/(2N)}\Big\}
\mathbb{E} \Big[ e_{j,N} \Big] \; . \nonumber
\end{eqnarray}
Denote the expressions appearing on the right hand side of  \eqref{lab33} by  $I_{N},I\!\!I_{N},I\!\!I\!\!I_{N}$ respectively.

{\em The term $I\!\!I\!\!I_{N}$.}
Using Taylor expansion for  $\exp\{-\theta^{2}\sigma^{2}/2N\}$
we can easily estimate
$
|I\!\!I\!\!I_{N}|\leq
C/N$ for some $C>0$ independent of $N$, so $\lim_{N\to+\infty}|I\!\!I\!\!I_{N}|=0$.

{\em The term $I\!\!I_{N}$.}
Fix   $\varepsilon>0.$ Then, there exists $C>0$ such that 
\begin{eqnarray}\label{lab34}
&&|I\!\!I_{N}|\leq 
\frac {C}{N}\sum_{j=0}^{N-1} \mathbb{E}\Big[ Z_{j+1}^2 \, \big| R_j\big|,  \,
|Z_{j+1} | \ge \varepsilon\sqrt{N}  \Big]\ \\
&+& \; \frac {C}{N}
\sum_{j=0}^{N-1} \mathbb{E}\Big[ Z_{j+1}^2 \, \big| R_j\big|,\, |Z_{j+1} | < \varepsilon\sqrt{N} \Big]=I\!\!I_{N1}+I\!\!I_{N2}\nonumber
\end{eqnarray}
Since  $|R_j|\le 1$ we have
 \begin{eqnarray*}
I\!\!I_{N1}\leq
 \frac {C}{N}\sum_{j=0}^{N-1} \mathbb{E}\Big[ Z_{j+1}^2, \,
| Z_{j+1} | \ge \varepsilon\sqrt{N}  \Big]\to0
 \end{eqnarray*}
as $N\to+\infty$, by virtue of M1).

 As for the second term on the utmost right hand side of \eqref{lab34} we can write that
 \begin{eqnarray*}
 && I\!\!I_{N2}\leq \frac{C}{N}\,\sup_{|h|<\varepsilon}|R(h)|\sum_{j=0}^{N-1}\mathbb{E}[\,\mathbb{E}[Z_{j+1}^{2}|\,\mathfrak{F}_{j}]]\\
 &&=C\,\sup_{|h|<\varepsilon}|R(h)|\,\mathbb{E}\left[\frac{<\!M\!>_{N}}{N}\right]
 \end{eqnarray*}
 Since $\sup_{|h|<\varepsilon}|R(h)|$ tends to $0$, as $\varepsilon\uparrow 0$ (see \eqref{012501}), using \eqref{M2b} we conclude that 
 \begin{equation}
 \label{062501}
 \limsup_{N\to+\infty}I\!\!I_{N2}<\frac{\rho}{2},
 \end{equation}
provided that $\varepsilon$ is chosen sufficiently small (independent of $N$). The value of $\rho>0$ appearing on the right hand side has been chosen at the beginning of the proof.

{\em The term $I_{N}$.}
To simplify notation we let  $\beta :=(\theta ^2\sigma ^2)/2.$
Fix  $K\geq 1$, and assume that $N=\ell K+r$, with $0\le r\le K-1$. Divide $\Lambda_N = \{0, \dots, N-1\}$ into
$\ell+1 $ blocks, $\ell$ of them of size $K,$ the last one of size $r$, i.e.
$
 \Lambda_N \;=\; \bigcup_{m=0}^{\ell-1} I_m, 
$
where $ I_m = \{m K , \dots, (m +1)K   -1\}$ for  $ m<\ell$ and 
$I_\ell = \{\ell K, \dots,\ell K+r\}$.
To simplify the consideration let us assume that all intervals $I_m$ (including the last one) have length $K$. 
 Then,
\begin{eqnarray}
 &&|I_{N}|\leq\frac{C}{N}e^{\beta/N}\Big| \sum_{m=0}^{\ell-1} \sum_{j\in I_m} e^{j\beta/N}
\mathbb{E}\Big [e_{j,N} \{ \sigma^2 - Z^{2}_{j+1} \} \Big]\Big|\label{lab316}\\
&&\leq\frac C N \Big| \sum_{m=0}^{\ell-1} e^{[(m-1)K+1]\beta/N}\sum_{j\in I_m}\big\{ e^{[j-(m-1)K]\beta/N}-1\big\}
\mathbb{E}\Big [e_{j,N} \{ \sigma^2 - Z^{2}_{j+1} \} \Big]\Big| \nonumber\\
&&+\frac C N e^{\beta/N}\Big|\sum_{m=0}^{\ell-1} \sum_{j\in I_m} e^{(m-1)K\beta/N}
\mathbb{E}\Big [e_{j,N} \{ \sigma^2 - Z^{2}_{j+1} \} \Big]\Big|\nonumber
\end{eqnarray}
Denote the two terms on the utmost right hand side of  \eqref{lab316} 
respectively by $I_{N,1}$ and $I_{N,2}$.
Since $|e^{x}-1|\leq Cx$ for all $x\in [0,1]$, letting $x=[j-(m-1)K]\beta/N$,
we get 
\begin{eqnarray*}
I_{N,1}\le \frac { CK}{N^2}\sum_{j=1}^N \left(1+\bbE Z_j^2\right)\to 0
\end{eqnarray*}
as $\ell\to+\infty$, in light of \eqref{sublinear}.

 As for the other term we can write
\begin{eqnarray}\label{lab35}
&&I_{N,2}\le  \frac C N \sum_{m=0}^{\ell-1}   \,
 \, \Big\vert \mathbb{E}\left[ \sum_{j\in I_m}e_{(m-1)K,N}
\{ \sigma^2 - \mathbb{E}[Z^{2}_{j+1}|\mathfrak{F_{j}}] \} \right]\Big\vert \,  \nonumber \\
&& +\frac C N \sum_{m=0}^{\ell-1} 
\mathbb{E}\, \Big\vert \sum_{j\in I_m}\!\! \Big\{ e_{j,N} - e_{(m-1)K,N}\Big\}
\{ \sigma^2\!\! -\!\! \mathbb{E}[Z^{2}_{j+1}|\mathfrak{F_{j}}] \} \Big\vert\, \nonumber.
\end{eqnarray}
The two expressions on the right hand side shall be denoted by $J_{N1}$ and  $J_{N2}$, respectively. Then,
\begin{eqnarray}
&&J_{N,1}\le\frac C\ell \sum_{m=0}^{\ell-1}   \,
\mathbb{E} \, \Big\vert  \sigma^2 \!\!-\!\! \frac 1K \bbE\left[<\!M\!>_{mK-1}-<\!M\!>_{(m-1)K}\Big|{\frak F}_{(m-1)K}\right]  \Big\vert.\nonumber
\end{eqnarray}
This expression tends to $0$, when $\ell\to+\infty$ and then subsequently $K\to+\infty$, by virtue of M2).

As for  $J_{N,2}$ it equals
\begin{equation}
\label{082501}
\frac C N \!\sum_{m=0}^{\ell-1} \!
\mathbb{E}\Big\vert\! \!\sum_{j\in I_m}\!\!\Big\{\![ e^{i(\theta/\sqrt{N})
  (M_{j}-M_{(m-1)K})}\!-\!1]e_{(m-1)K,N} \!\Big\}\!
\{ \sigma^2 \!\!-\!\! \mathbb{E}[Z^{2}_{j+1}|\mathfrak{F_{j}}] \} \Big\vert.
\end{equation}
Consider two events:  $F:=[|( M_{j}-M_{(m-1)K})/\sqrt{N}|<\varepsilon]$ and its complement $F^c:=[| (M_{j}-M_{(m-1)K})/\sqrt{N}|\geq\varepsilon] $ and split the integration accordingly. We obtain two terms $L_{N,1}$, $L_{N,2}$ depending on whether we integrate over $F$, or $F^c$ respectively.  Using a well known estimate $|e^{i\varepsilon}-1|\leq\varepsilon$  we get
\begin{eqnarray*}
&&
L_{N,1}\leq  \frac{ C\varepsilon}{N} \sum_{j=1}^{N} (1+\bbE Z_{j+1}^2). 
\end{eqnarray*}
As a result of \eqref{sublinear} we conclude that
\begin{equation}
\label{092501}
\limsup_{\ell\to+\infty}L_{N,1}<\frac{\rho}{2},
\end{equation}
provided that $\rho>0$ is sufficiently small.
In the other case we get ($N=\ell K$)
\begin{eqnarray*}
&&
L_{N,2}\leq \frac{C}{\ell K}\sum_{m=0}^{\ell-1} \sum_{j=(m-1)K}^{mK-1}\mathbb{E}[1+Z_{j+1}^2,\, |M_{j}-M_{(m-1)K}|\geq\varepsilon\sqrt{\ell K}]
\end{eqnarray*}
and using \eqref{m3} we conclude that
$$
\lim_{K\to+\infty}\limsup_{\ell\to+\infty}L_{N,2}=0.
$$
The above argument allows us to conclude that if $N=\ell K+r$ for some $0\le r\le K-1$, then
$$
\limsup_{K\to+\infty}\limsup_{\ell\to+\infty}\left|e^{(\theta^2\sigma^2)/2}\,\mathbb{E}\,e^{i(\theta\sqrt{N})M_{N}}-1\right|<\rho
$$
for any $\rho>0$. This of course implies the desired formula \eqref{052501}.

\bigskip

\subsection*{Acknowledgements} Both authors wish to express their thanks to prof. T. Szarek for pointing out reference \cite{guivarch}. They are also profoundly grateful to an anonymous referee of the paper for careful reading of the manuscript and very helpful remarks that lead to its significant improvement.

\end{document}